\documentclass[10pt]{amsart}

\usepackage{amsmath, graphicx, cite, enumerate, comment}
\usepackage{amssymb, amsthm, amscd}
\usepackage{latexsym}
\usepackage{amssymb}
\usepackage[english]{babel}
\usepackage{amsmath,amssymb}
\usepackage{amsfonts}
\usepackage{cite}
\usepackage{marvosym}
\usepackage{mathrsfs}
\usepackage{amsthm}
\usepackage{hyperref}
\usepackage{verbatim}

\usepackage{tikz} \usetikzlibrary{arrows,shapes,patterns} 
\usepackage{lipsum} 

\newtheorem{thm}{Theorem}
\newtheorem{lem}[thm]{Lemma}

\newtheorem{cor}[thm]{Corollary}
\newtheorem{prop}[thm]{Proposition}
\newtheorem{definition}[thm]{Definition}

\theoremstyle{remark}
\newtheorem{rmk}[thm]{Remark}

\theoremstyle{definition}

\numberwithin{equation}{section}

\newcommand{\R}{\mathbb{R}}

\newcommand{\be}{\begin{equation}}
\newcommand{\ee}{\end{equation}}

\newcommand{\ben}{\begin{equation*}}
\newcommand{\een}{\end{equation*}}

\def\barroman#1{\sbox0{#1}\dimen0=\dimexpr\wd0+1pt\relax
  \makebox[\dimen0]{\rlap{\vrule width\dimen0 height 0.06ex depth 0.06ex}%
    \rlap{\vrule width\dimen0 height\dimexpr\ht0+0.03ex\relax 
            depth\dimexpr-\ht0+0.09ex\relax}%
    \kern.5pt#1\kern.5pt}}

\newcommand{\twopartdef}[4]
{
	\left\{
		\begin{array}{ll}
			#1 & \mbox{if  } #2 \bigskip \\
			#3 & \mbox{if  } #4
		\end{array}
	\right.
}

\def\XXint#1#2#3{{\setbox0=\hbox{$#1{#2#3}{\int}$}
     \vcenter{\hbox{$#2#3$}}\kern-.5\wd0}}

\newcommand{\N}{\mathbb{N}}
\newcommand{\h}{\mathscr{H}}

\newcommand{\emb}{\hookrightarrow}
\newcommand{\cemb}{\subset \subset}

\newcommand{\id}{\ d}

\newcommand{\D}{{\nabla}}

\newcommand{\eps}{{\varepsilon}}

\newcommand{\floor}[1]{\lfloor #1 \rfloor}

\def\XXint#1#2#3{{\setbox0=\hbox{$#1{#2#3}{\int}$}
     \vcenter{\hbox{$#2#3$}}\kern-.5\wd0}}

\begin{document} 	               
       \thanks{\textsl{Mathematics Subject Classification (MSC 2010):} Primary 53A10; Secondary 53C42, 49Q05.}
     \title[Bubbling analysis for free boundary minimal surfaces]{Bubbling analysis and geometric convergence results for free boundary minimal surfaces}
     \author{Lucas Ambrozio, Reto Buzano, Alessandro Carlotto and Ben Sharp}
     \address{ \noindent L. Ambrozio: Department of Mathematics, University of Warwick, Coventry CV4 7AL, United Kingdom, \textit{E-mail address: L.Coelho-Ambrozio@warwick.ac.uk}
     	\newline \newline 
     	\indent R. Buzano: School of Mathematical Sciences, Queen Mary University of London, London E1 4NS, United Kingdom, \textit{E-mail address: r.buzano@qmul.ac.uk}
     	\newline\newline
     	\indent A. Carlotto:  Department of Mathematics, ETH, 8092 Z\"urich, Switzerland, \textit{E-mail address: alessandro.carlotto@math.ethz.ch}
     	 \newline \newline 
     \indent B. Sharp: School of Mathematics,
     University of Leeds, Leeds LS2 9JT, United Kingdom \textit{E-mail address: B.G.Sharp@leeds.ac.uk}}

      	\begin{abstract} We investigate the limit behaviour of sequences of free boundary minimal hypersurfaces with bounded index and volume, by presenting a detailed blow-up analysis near the points where curvature concentration occurs. Thereby, we derive a general quantization identity for the total curvature functional, valid in ambient dimension less than eight and applicable to possibly improper limit hypersurfaces. In dimension three, this identity can be combined with the Gauss-Bonnet theorem to provide a constraint relating the topology of the free boundary minimal surfaces in a converging sequence, of their limit, and of the bubbles or half-bubbles that occur as blow-up models. We present various geometric applications of these tools, including a description of the behaviour of index one free boundary minimal surfaces inside a 3-manifold of non-negative scalar curvature and strictly mean convex boundary. In particular, in the case of compact, simply connected, strictly mean convex domains in $\R^3$ unconditional convergence occurs for all topological types except the disk and the annulus, and in those cases the possible degenerations are classified.
      	\end{abstract}

      	\maketitle   
      	
  
     	\section{Introduction}\label{sec:intro}
     	
     	In the last decade we have witnessed various significant developments in the study of free boundary minimal hypersurfaces, new conceptual links have emerged and diverse tools have been employed to produce novel existence results in different geometric settings. The expansion of this research direction has brought back attention to a number of classical open problems in the field and has motivated the extension of various foundational results to this setting.
     	
     Differently from the case of closed minimal surfaces, the free boundary theory is already very rich in the context of compact subdomains of $\R^3$, in fact even in the special but important instance of the unit ball: various constructions have been presented by Fraser and Schoen \cite{FS16} (genus zero and any number of boundary components), by Folha, Pacard and Zolotareva \cite{FPZ15} (genus zero or one and any sufficiently large number of boundary components), by Ketover \cite{Ket16A} and Kapouleas and Li \cite{KL17} (arbitrarily large genus and three boundary components) and by Kapouleas and Wiygul \cite{KW17} (arbitrarily large genus and one boundary component). In higher dimension, infinite families of examples have been found, in the Euclidean unit ball, by Freidin, Gulian and McGrath \cite{FGM16}. On the other hand, more general constructions like the min-max \`a la Almgren-Pitts or the degree-theoretic approach \`a la White have led to further results that are widely applicable: in that respect we shall mention here the results by Li \cite{Li15}, Li-Zhou \cite{LZ16B}, De Lellis-Ramic \cite{DR16} and Maximo-Nunes-Smith \cite{MNS13}. To those developments, one should add the older works that mostly appeal to the parametric approach (see \cite{Cou40, Cou50, Str84, GJ86A, Jo86, Jo86b, Jo87, Wan99, Fra00}  and references therein, among others). \\

     	Motivated by all these examples, we wish to proceed here in the compactness analysis of free boundary minimal hypersurfaces that was presented in \cite{ACS17}. In their earlier work \cite{FL14}, Fraser and Li proved that the set of free boundary minimal surfaces of fixed topological type is strongly compact in any Riemannian 3-manifold with non-negative Ricci curvature and convex boundary, where one considers smooth graphical convergence with multiplicity one. This sort of conclusion cannot be expected in higher dimension, or when the curvature assumptions are relaxed, and thus one needs to approach the problem from a different perspective. \\

     	First of all, we recall some basic definitions and notation. Given a compact Riemannian manifold $(N^{n+1},g)$ of dimension $n+1\geq 3$, with non-empty boundary $\partial N$, let $M^n$ be a codimension one properly embedded submanifold: we say that $M^n$ is a free boundary minimal hypersurface in $(N^{n+1},g)$ if it is a critical point for the $n$-dimensional area functional under variations satisfying the sole constraint that $\partial M\subset\partial N$ or, equivalently, if $M$ has zero mean curvature and meets the ambient boundary orthogonally. In order to avoid ambiguities, let us remark that properness is always understood here in the strong sense that $M \cap \partial N=\partial M$ (which is consistent with \cite{FL14, LZ16B, ACS17, ACS18b} among others).
    We let $\mathfrak{M}$ denote the class of smooth, connected, and properly embedded free boundary minimal hypersurfaces and consider the subclasses given by
     	\[
     		\mathfrak{M}(\Lambda,I) := \{ M \in \mathfrak{M} \ : \ \h^n(M) \leq \Lambda,\ index(M) \leq I\}
     			\]
     			so with the additional requirements that area and Morse index be bounded from above (cf. \cite{Sha15}); and for $p$ an integer greater or equal than one
     \[
     			\mathfrak{M}_p(\Lambda,\mu) := \{ M \in \mathfrak{M} \ : \ \h^n(M) \leq \Lambda,\ \lambda_p \geq -\mu\}
     			\]
     			so with area bounded from above and $p^{th}$ eigenvalue $\lambda_p$ of the Jacobi operator bounded from below (cf. \cite{ACS15}).

     		In \cite{ACS17}, Ambrozio-Carlotto-Sharp proved a compactness theorem for $\mathfrak{M}_p(\Lambda,\mu)$ when the ambient dimension is less than eight: given a sequence of free boundary minimal hypersurfaces $\left\{M_k\right\}$ in $\mathfrak{M}_p(\Lambda,\mu)$, there is some smooth limit to which the sequence sub-converges smoothly and graphically (with finite multiplicity $m$) away from a discrete set $\mathcal{Y}$ on the limit, where curvature concentration occurs. These results apply, as a special case, to all classes of the form $\mathfrak{M}(\Lambda,I)$ (since this is clearly the same as $\mathfrak{M}_{I+1}(\Lambda,0)$). 
     			A significant part of the present article is aimed at an accurate description of the local picture around those points of bad convergence, in analogy with the results obtained, for the closed case, by Buzano and Sharp in \cite{BS17}. We then specify these results to the three dimensional scenario, and derive various sorts of new geometric results based on these tools.	\\

     		As a first step in this program, we quantify the lack of smooth compactness (via a blow-up argument) by a finite number of non-trivial, complete, properly embedded minimal hypersurfaces of finite total curvature $\Sigma^n\emb \R^{n+1}$ as in \cite{BS17}, except this time it is possible that only `half' of $\Sigma$ is captured at a boundary point of bad convergence. In order to formalize this picture, let us introduce some terminology:
     		we employ the word \textsl{bubble} to denote a complete, connected and properly embedded minimal hypersurface of finite total curvature in $\R^{n+1}$; we employ the word \textsl{half-bubble} to denote a complete, connected, properly embedded minimal hypersurface that is contained in a half-space and has (non-empty) free boundary with respect to the boundary of this half-space,  and has finite total curvature. In both cases, the total curvature is understood to be the integral of the $n$-th power of the length of the second fundamental form, the corresponding functional being denoted by $\mathcal{A}(\cdot)$. Furthermore, we shall say that a (half-)bubble is \textsl{non-trivial} if it is not flat, namely if it is not a (half-)hyperplane.
     		
     		Roughly speaking, the blow-up limits one obtains are non-trivial bubbles at interior points of bad convergence, while at boundary points of bad convergence one may get either non-trivial bubbles or non-trivial half-bubbles. This assertion is formalized and expanded in the following two statements, the first one (Theorem \ref{thm:quant}) collecting the global outcome and the second one (Theorem \ref{thm:bubbles}) describing the local picture. 
     		
     		\begin{thm}\label{thm:quant}
     			Let $2\leq n\leq 6$ and $(N^{n+1},g)$ be a compact Riemannian manifold with boundary. For fixed $\Lambda, \mu\in \R_{\geq 0}$ and $p\in \N_{\geq 1}$, suppose that $\{M_k\}$ is a sequence in $\mathfrak{M}_p(\Lambda, \mu)$. Then there exist a smooth, connected, compact embedded minimal hypersurface $M\subset N$ meeting $\partial N$ orthogonally along $\partial M$, $m\in \mathbb{N}$ and a finite set $\mathcal{Y}\subset M$ with cardinality $|\mathcal{Y}|\leq p-1$  such that, up to subsequence, $M_k \to M$ locally smoothly and graphically on $M\setminus \mathcal{Y}$ with multiplicity $m$.  
     			Moreover there exists a finite number of non-trivial bubbles or half-bubbles $\{\Sigma_j\}_{j=1}^J$ with $J\leq p-1$ and
     			\begin{align*}
     			\mathcal{A}(M_k) &\to m\mathcal{A}(M) + \sum_{j=1}^J \mathcal{A}(\Sigma_j), \ \ \ (k\to\infty).
     			\end{align*}
     			For $k$ sufficiently large, the hypersurfaces $M_k$ of this subsequence are all diffeomorphic to one another. Finally, if $M\in\mathfrak{M}$ then $M\in\mathfrak{M}_{p}(\Lambda,\mu)$. 
     		\end{thm}
     		
     		\begin{rmk}\label{rem:properness}
     		Concerning the very last clause, we shall recall that the class $\mathfrak{M}_p(\Lambda,\mu)$ (and, more generally, the whole class $\mathfrak{M}$) is in general not closed under smooth graphical convergence with multiplicity one: easy examples of non-convex domains in $\R^3$ show that the limit of elements in $\mathfrak{M}$ may be not properly embedded, and in fact have a large contact set with the boundary of the ambient manifold. Following \cite{ACS17}, we introduce the following general assumption:
     		
     		\begin{center}
     			$(\textbf{P})$ \ \ if $M\subset N$  has zero mean curvature and meets the boundary of the ambient manifold orthogonally along its own boundary, then it is proper.
     		\end{center}
     		For instance, this condition is implied by the geometric requirement that $\partial N$ is mean convex (i.e. $H\geq 0$) and has no minimal components, which is condition $(\textbf{C})$ in \cite{ACS17}.
     		If we assume that $(N^{n+1},g)$ satisfies assumption $(\textbf{P})$, then the limit $M$ must be properly embedded, and thus $M\in\mathfrak{M}_{p}(\Lambda,\mu)$ always holds. Without this assumption, an improper contact set with $\partial N$ may occur, and even the definition of Morse index becomes delicate and much less canonical than in the proper case (for a discussion of the improper case, motivated by min-max constructions, see \cite{GZ18}).
     		\end{rmk}	
     		
     		Let us now add some comments on the significance and straightforward geometric implications of Theorem \ref{thm:quant}.
     		In ambient dimension three (corresponding to $n=2$), the total curvature of any bubble is an integer multiple of $8\pi$ (cf. \cite{Oss63, Oss64}) and thus the total curvature of any half-bubble is an integer multiple of $4\pi$: hence Theorem \ref{thm:quant} implies that for a sequence of surfaces that eventually satisfy $\mathcal{A}(M_k)\leq 4\pi-\delta$ for some $\delta>0$, the set $\mathcal{Y}$ must be empty and the convergence to $M$ is smooth and graphical everywhere (but possibly with higher multiplicity, which however will not happen if the limit is two-sided).  \\
     		
     		As a second direct application, we can prove a uniform bound on the total curvature, and of the Morse index, of any given set $\mathfrak{M}_p(\Lambda,\mu)$:
     		
     		\begin{cor}\label{cor:totcurvbound}
     			Let $2\leq n\leq 6$ and $(N^{n+1},g)$ be a compact Riemannian manifold with boundary. Given $\Lambda, \mu\in \R_{\geq 0}$ and $p\in\N_{\geq 1}$ there exist:
     			\begin{enumerate}
     			 \item{a constant $C=C(p,\Lambda,\mu, N, g)$ such that the total curvature of any element in $\mathfrak{M}_p(\Lambda, \mu)$ is bounded from above by $C$.}
     			 \item{ a constant $I=I(p,\Lambda,\mu, N, g)$ such that the Morse index of any element in $\mathfrak{M}_p(\Lambda, \mu)$ is bounded from above by $I$, so that $\mathfrak{M}_p(\Lambda,\mu)\subset\mathfrak{M}(\Lambda,I)$. }
     			 \end{enumerate}
     		\end{cor}	
     		
     		Lastly, a key point in the statement of Theorem \ref{thm:quant} is that the topology (in fact: the diffeomorphism type) of the hypersurfaces in the sequence $\left\{M_k\right\}$ eventually stabilizes. Following the notation of \cite{ACS17}, given two smooth manifolds $M_1, M_2$ (possibly with non-empty boundary), we write $M_1\simeq M_2$ if they are diffeomorphic, and for a subset $\mathfrak{S}\subset\mathfrak{M}$ we let $\mathfrak{S}/\simeq $ denote the set of corresponding equivalence classes modulo diffeomorphisms.
     		
     		\begin{cor}\label{cor:finite}
     			Let $2\leq n\leq 6$ and $(N^{n+1},g)$ be a compact Riemannian manifold with boundary. Given $\Lambda, \mu\in \R_{\geq 0}$ and $p\in \N_{\geq 1}$ the quotient $\mathfrak{M}_p(\Lambda, \mu)/\simeq$ is finite.
     		\end{cor}	
     		
     		This is the unconditional, general counterpart of Corollary 7 in \cite{ACS17}, which was proven to hold true under suitable geometric assumptions on the ambient manifold, namely requiring that the Ricci curvature of $(N,g)$ be non-negative and $\partial N$ be strictly convex or that the Ricci curvature of $(N,g)$ be non-negative and $\partial N$ be convex (and strictly mean convex). The fact that one can get rid of extra hypotheses relies, roughly speaking, on the fact that we can get a good understanding of the structure of each hypersurface $M_k$ near the points where its curvature is large, at least for $k$ big enough.
     	The last assertion is justified by the aforementioned local description result:

     	\begin{thm}\label{thm:bubbles}
     		With the setup as in Theorem \ref{thm:quant}, for each $y\in \mathcal{Y}$ there exist a finite number of point-scale sequences $\{(p^i_k, r^i_k)\}_{i=1}^{J_y}$ where $\sum_{y\in \mathcal{Y}} J_y \leq p-1$ with $M_k \ni p^i_k \to y$, $r^i_k\to 0$, and finite numbers of non-trivial bubbles and half-bubbles $\{\Sigma_i\}_{i=1}^{J_y}$, such that the following is true.
     		\begin{itemize}
     			\item For all $i\neq j$, we have 
     			\begin{equation*}
     			\frac{r^i_k}{r^j_k} + \frac{r^j_k}{r^i_k} + \frac{dist_g(p^i_k, p^j_k)}{r^i_k + r^j_k} \to \infty.
     			\end{equation*}
     			Taking normal coordinates centered at $p^i_k$, then $\widetilde{M}^i_k:=\frac{M_k}{r^i_k}$ converges locally smoothly and graphically, away from the origin, to a disjoint union of (half-)hyperplanes and at least one non-trivial bubble or half-bubble.  The convergence to any non-trivial component of the limit occurs with multiplicity one. 
     			\item Given any other sequence $M_k \ni q_k$ and $\varrho_k \to 0$ with $q_k \to y$ and 
     			\begin{equation*}
     			\min_{i=1,\dots J_y} \Big(\frac{\varrho_k}{r^i_k} + \frac{r^i_k}{\varrho_k} + \frac{dist_g(q_k,p^i_k)}{\varrho_k + r^i_k}\Big)\to \infty
     			\end{equation*}
     			then taking normal coordinates at $q_k$, we obtain that $\widehat{M}_k:= \frac{M_k}{\varrho_k}$ converges to a collection of parallel (half-)hyperplanes. 
     		\end{itemize}
     		When $n=2$, any blow-up limit of $\widetilde{M}^i_k$ is always connected. The convergence is locally smooth, and of multiplicity one. Moreover we always have 
     			\[
     		  (*) \ \ \ \	\ \	\frac{dist_g(p^i_k,p^j_k)}{r^i_k + r^j_k}\to \infty .
     			\]
     		\end{thm}
     		Notice that condition $(*)$ has, in fact, a transparent geometric interpretation: it ensures that one can \textsl{separate} the bubble regions, so that in a certain sense there is no interaction between different regions of high curvature. More precise corollaries of this fact will be discussed later.
     		
     		\begin{rmk}\label{rem:lima}
     			Thanks to recent work \cite{Lim17} by V. Lima, extending to the free boundary setting the results by Eijiri-Micallef \cite{EM08} and Cheng-Tysk \cite{CT94}, we know that \emph{a}) a uniform bound both on the area and on the topology of a sequence of orientable free boundary minimal surfaces implies a uniform bound on the Morse index, and \emph{b}) when $n\geq 3$ a uniform bound on the area and on the total curvature implies a uniform bound on the Morse index. It follows that both our main theorems can be rephrased with those assumptions, instead. Moreover, we note that when $n=2$ the theorem by V. Lima can be regarded as a partial converse to the results in \cite{ACS18b}, where the Morse index is proven to be bounded from below by an affine function of the genus and the number of boundary components of the surface in question, under suitable curvature conditions on the ambient manifold.
     		\end{rmk}

     	   For the rest of this introduction, let us focus on the case of ambient dimension three. In that case, one can rely on the Gauss-Bonnet theorem, and on the varifold convergence of the boundaries (see Proposition \ref{prop:boundaryconv} and Corollary \ref{cor:geodcurv} for precise statements) to rewrite the quantization identity in the form 
     	\[
     	\chi(M_k) = m\chi(M) + \frac{1}{2\pi} \sum_{j=1}^J \int_{\Sigma_j} K_{\Sigma_j} \id \h^2,
     	\]
     		which holds true, with the setup as in Theorem \ref{thm:quant}, for all $k$ sufficiently large. Actually, the second summand on the right-hand side can also be expressed only in terms of topological data (see Subsection \ref{subs:geomhalf} for a detailed discussion), so that we can derive the formula we will employ in all of our applications:
     		
     		\begin{cor}\label{cor:chi}
     			In the setting of Theorem \ref{thm:quant}, specified to $n=2$, we have for all $k$ sufficiently large
     			\begin{equation*}\label{eq:topfundsymINTRO}
     			\chi(M_k)= m\chi(M)+ \sum_{j=1}^J (\chi(\Sigma_j)-b_j),
     			\end{equation*}
     			where $\chi(\Sigma_j)$ denotes the Euler characteristic of $\Sigma_j$ and $b_j$ denotes the number of its ends.
     		\end{cor}
     		
     		This is the starting point for our primary geometric applications. We present three instances, which are meant to illustrate the method, and leave other possible extensions in the form of remarks.\\

     		Here is the first application we wish to discuss: since we can fully classify bubbles and half-bubbles of Morse index less than two (see Corollary \ref{cor:stable} and Corollary \ref{cor:indexone}) we can then get novel, unconditional, geometric convergence results for sequences of free boundary minimal surfaces of low index. Furthermore, we can specialize the general blow-up analysis presented in Theorem \ref{thm:bubbles} to give an accurate description of the possible degenerations that may occur when the convergence is not smooth everywhere.

     			\begin{thm}\label{thm:conv1}
     				Let $(N^3,g)$ be a compact Riemannian manifold, with non-empty boundary $\partial N$. Assume that: 
     				\begin{itemize} 
     					\item[\emph{a)}]{\underline{\textsl{either}} the scalar curvature of $(N,g)$ is positive and $\partial N$ is mean convex with no minimal component;}
     					\item[\emph{b)}]{\underline{\textsl{or}} the scalar curvature of $(N,g)$ is non-negative and $\partial N$ is strictly mean convex.}
     				\end{itemize}
     				Then, for any $\Lambda>0$ the following assertions hold:
     				\begin{enumerate}
     				\item{The class $\mathfrak{M}(\Lambda,0)$ is sequentially compact in the sense of smooth multiplicity one convergence. Similarly,  any subclass of $\mathfrak{M}(\Lambda, 1)$ of fixed topological type is sequentially compact, in the sense of smooth multiplicity one convergence, for all given topological types except those of the disk and of the annulus. In particular, we obtain unconditional sequential compactness for any class of non-orientable surfaces of given topological type.}
     				\item{Let $\left\{M_k\right\}$ be a sequence of disks (respectively: annuli) in $\mathfrak{M}(\Lambda, 1)$. Then: \textsl{either} a subsequence converges smoothly, with multiplicity one, to an embedded minimal disk (respectively: annulus) of index at most one \textsl{or} there exists a subsequence converging smoothly, with multiplicity two and exactly one vertically cut catenoidal half-bubble as per Definition \ref{def:horcurvercut} (respectively:  exactly one catenoidal bubble), to a properly embedded, free boundary stable minimal disk. As a result, if $N$ contains no stable, embedded, minimal disks then strong compactness holds.}
     				\end{enumerate}
     				All conclusions still hold true without assuming any a priori upper area bound if $N$ is simply connected and, in case b), if moreover there is no \textsl{closed} minimal surface in $N$. 
     			\end{thm}
     			
     			Similar results can be obtained for sequences whose Morse index is bounded from above by any given integer $k\in\mathbb{N}$, and whose area is also uniformly bounded. On the other hand, we can also prove strong compactness theorems for sequences of free boundary minimal surfaces that satisfy certain quantitative lower bounds on either their area or on the length of their boundaries. Roughly speaking, this relies on the fact that for 3-manifolds of positive scalar curvature and mean convex boundary we have area estimates for the possible stable limits, as per Lemma \ref{lem:area}, and an effective multiplicity estimate in terms of the number of ends of the bubbles and half-bubbles that occur as blow-up models (cf. Proposition 13 in \cite{ABCS18}).

     				\begin{thm}\label{thm:conv2}
     					Let $(N^3,g)$ be a compact Riemannian manifold, with non-empty boundary $\partial N$, and let $\left\{M_k\right\}$ be a sequence in $\mathfrak{M}(\Lambda,1)$ for some $\Lambda>0$.
     					\begin{enumerate}
     						\item{Assume that the scalar curvature of $(N,g)$ is bounded from below by $\varrho>0$ and $\partial N$ is mean convex with no minimal component. If
     							\[
     							\limsup_{k\to\infty}\h^2(M_k)>\frac{8\pi}{\varrho}
     							\]
     							then, up to extracting a subsequence, $\left\{M_k\right\}$ converges smoothly to some element of $\mathfrak{M}(\Lambda,1)$ with multiplicity one.}	
     						\item{Assume that the scalar curvature of $(N,g)$ is non-negative and the mean curvature of $\partial N$ is bounded from below by $\sigma>0$. If 
     							\[
     							\limsup_{k\to\infty}\h^1(\partial M_k)>\frac{4\pi}{\sigma}
     							\]
     							then, up to extracting a subsequence, $\left\{M_k\right\}$ converges smoothly to some element of $\mathfrak{M}(\Lambda,1)$ with multiplicity one.}	
     					\end{enumerate}	
     					 All conclusions still hold true without assuming any a priori upper area bound if $N$ is simply connected and, in case (2), if moreover there is no \textsl{closed} minimal surface in $N$. 
     				\end{thm}

     		When considering sequences of free boundary minimal surfaces of bounded index and area, one may witness some `loss of topology' in the limit: Theorem \ref{thm:quant}, and in particular Corollary \ref{cor:chi} can further be employed to fully understand and quantify this loss. Let us focus, for simplicity, on the case where the ambient is orientable and all free boundary minimal surfaces are orientable as well (which occurs, for instance, in simply connected Euclidean domains).
     Given a (half-)bubble $\Sigma$ we let the constant $\delta(\Sigma)$ be defined by the equation
      \[
      \chi(\Sigma)=2-2 \delta(\Sigma)-b(\Sigma)
      \]
      where $b(\Sigma)$ denotes the number of ends of $\Sigma$. Notice that when $\Sigma$ is a full bubble then $\delta(\Sigma)$ equals the genus of $\Sigma$ (cf. equation \eqref{eq:GBcomplete}), but this is patently not the case for half-bubbles: for instance if $\Sigma$ is a vertically cut half-catenoid (see Definition \ref{def:horcurvercut}) then $\chi(\Sigma)=1, b(\Sigma)=2$ and $\delta(\Sigma)=-1/2$ (which also shows that $\delta(\cdot)$ is not integer-valued, and may be negative). 
      
      \begin{thm}\label{thm:lsc}
      	Let $(N^{3},g)$ be a compact, orientable Riemannian manifold with non-empty boundary $\partial N$.
      	Consider a sequence of orientable, embedded, free boundary minimal surfaces $\{M_k\}\subset \mathfrak{M}_p(\Lambda,\mu)$ for some fixed constants $\Lambda\in \R$, $\mu\in \R$ and $p\in\N_{\geq 1}$ independent of $k$, and assume it has an orientable limit
      	$M\in \mathfrak{M}_p(\Lambda,\mu)$, in the sense of smooth graphical convergence with multiplicity $m\geq 1$ away from a finite set $\mathcal{Y}$ of points. 
      	Then for all sufficiently large $k\in\mathbb{N}$ one has
      	\begin{equation*}\label{eq:lsc}
      	2m\cdot  genus(M)+m\cdot boundaries(M)
      	+2\sum_{j=1}^J\delta(\Sigma_j)
      	\leq 2 \cdot genus(M_k)+boundaries(M_k).
      	\end{equation*}
      	The inequality above is strict unless
      	\[
      	m=1+\sum_{j=1}^J(b_j-1).
      	\]
      	Under assumption (\textbf{P}) the number of non-trivial bubbles plus half-bubbles is at most $m-1$, and if it equals $m-1$ then each bubble is a catenoid and each half-bubble is a vertically cut half-catenoid.
      \end{thm}
      
      Let us conclude this introduction with a brief description of the structure of the present article. In Section \ref{sec:prelim} we present some preliminary results concerning half-bubbles: after a general discussion relating their Morse indices and topology to those of their doubles, we employ these facts to derive various classification results of independent interest. Section \ref{sec:bubbling} and \ref{sec:neck} are devoted to the bubbling and neck analysis, respectively, and lead to a complete proof of Theorem \ref{thm:quant} and Theorem \ref{thm:bubbles} above. Lastly, the global geometric applications are then collected in Section 5.

     		\textsl{Acknowledgements.}  During the preparation of this article, L. A. was supported by the EPSRC on a Programme Grant entitled `Singularities of Geometric Partial Differential Equations' reference number EP/K00865X/1.
     		This project was completed while A. C. was visiting the Mathematisches Forschungsinstitut Oberwolfach, and he would like to thank the director Gerhard Huisken and the staff members for the warm hospitality and excellent working conditions.

        \section{Preliminary results on half-bubbles}\label{sec:prelim}

        	Let $\Sigma^n\subset \R^{n+1}$ be a half-bubble in the sense above. Let us denote by $\Pi_1$ the closed half-space bounded by $\Pi$ that contains the interior of $\Sigma$ and by $\Pi_2$ the other closed half-space bounded by $\Pi$. The terminology \textsl{half-bubble} can be justified as follows: if we reflect $\Sigma$ across $\Pi$ in $\R^{n+1}$ we get a minimal hypersurface without boundary, which a priori is only $C^1$, but a posteriori is $C^{1,\alpha}$ by means of a standard application of De Giorgi-Nash estimates, hence smooth by Schauder theory. Such a minimal hypersurface shall be denoted by $\check{\Sigma}$ and be referred to as the \textsl{double} of $\Sigma$.

        	In relation to the geometric applications we are about to present, we need to compare the Morse index of $\Sigma$ (as a free boundary minimal surface, thus with suitable boundary conditions) and the Morse index of $\check{\Sigma}$. In fact, our results will follow as a specification of a more general discussion.

        	\subsection{Schr\"odinger-type operators on involutive manifolds}\label{subs:general}
        	
        	Let $(\check{\Sigma}^{n},g)$ be a complete Riemannian manifold, without boundary (in Subsection \ref{subs:morsehalf} we will then specify our discussion to the ambient manifold $\check{\Sigma}$ presented in the previous paragraph, namely a bubble with its induced metric). Suppose there exists a Riemannian involution $\tau:(\check{\Sigma},g) \to (\check{\Sigma},g)$, namely a smooth isometry satisfying the two conditions:
        	\[
        	\tau\circ\tau = id, \ \ \tau\neq id.
        	\] 
        	
        Given a linear functional space $X$ consisting of functions defined on $(\check{\Sigma},g)$ let us then introduce the subspaces of even and odd functions with respect to the action of $\tau$:
        	\[
        	X_\mathcal{E}=\left\{\varphi\in X: \varphi\circ\tau=\varphi \right\}, \ X_\mathcal{O}=\left\{\varphi\in X: \varphi\circ\tau=-\varphi \right\}.
        	\]	
        	In particular, notice that 
        	\[
        	C^{\infty}=C^{\infty}_{\mathcal{E}}\oplus C^{\infty}_{\mathcal{O}}.
        	\]
        	This direct sum is orthogonal in $L^2$ when restricted to the subspace of smooth, compactly supported functions. In fact, we have
        	\[
        	L^2=L^2_{\mathcal{E}}\oplus^{\perp} L^2_{\mathcal{O}}.
        	\]
        Given a function $V\in C^{\infty}_{\mathcal{E}}$ we wish to study a Schr\"odinger-type operator of the form
        	\[
        	T_V\varphi=\Delta_g\varphi+V\varphi, 
        	\]
        	together with the associated quadratic form
        	\[
        	Q_V(\varphi,\varphi)=-\int_{\check{\Sigma}}\varphi T_V \varphi\,d\h^n=\int_{\check{\Sigma}}(|\nabla \varphi|^2-V\varphi^2)\,d\h^n
        	\]
        	which a priori is defined on the set of smooth, compactly supported functions.

        	\begin{definition}\label{def:index} In the setting above, we define
        		\[
        		Ind (Q_V), \ \ Ind_{\mathcal{E}} (Q_V), \ \ Ind_{\mathcal{O}} (Q_V)
        		\]
        		as the largest dimension of a linear subspace of 
        		\[
        		C^{\infty}_c, \ \ (C^{\infty}_c)_{\mathcal{E}}, \ \ (C^{\infty}_c)_{\mathcal{O}}
        		\]
        		respectively, where the quadratic form $Q_V$ is negative definite.
        	\end{definition}

        	Thereby, the following inequality is straightforward:
        	
        	\begin{lem}\label{lem:ineq}
        		In the setting above, 
        		\[
        		Ind (Q_V) \geq Ind_{\mathcal{E}} (Q_V)+Ind_{\mathcal{O}} (Q_V).
        		\]
        		and
        		\[
        		Ind (Q_V)=0 \Leftrightarrow Ind_{\mathcal{E}} (Q_V)=0 \ \text{and} \ Ind_{\mathcal{O}} (Q_V)=0.
        		\] 
        	\end{lem}
        	
        	\begin{proof} Since $V$ is even and $\tau$ is an isometry, it is immediate to check that the operator $T_V$ preserves the decomposition of $C^{\infty}$ into even and odd functions. Thus, if $\varphi_{\mathcal{E}}\in (C^{\infty}_c)_{\mathcal{E}}$ and $\varphi_{\mathcal{O}}\in (C^{\infty}_c)_{\mathcal{O}}$ then
        		\begin{equation}\label{eq:mixedprod}
        		Q_{V}(\varphi_{\mathcal{E}},\varphi_{\mathcal{O}})=0,
        		\end{equation}
        		and hence, by bilinearity, given any $\varphi\in C^{\infty}_c$ and writing $\varphi=\varphi_{\mathcal{E}}+\varphi_{\mathcal{O}}$ one has
        		\[
        		Q_{V}(\varphi,\varphi)=Q_{V}(\varphi_{\mathcal{E}},\varphi_{\mathcal{E}})+Q_{V}(\varphi_{\mathcal{O}},\varphi_{\mathcal{O}}),
        		\]
        		which easily implies the claims.
        	\end{proof}

        	In fact, we claim that the inequality above is actually an equality. 
        	With that goal in mind, we restrict to the case of \textsl{finite} Morse index and recall a simple but useful result:
        	
        	\begin{prop}\label{pro:l2spec}(cf. \cite{FC85} Proposition 2)
        		In the setting above, the following two statements are equivalent:
        		\begin{enumerate}
        			\item{the index of the quadratic form $Q_V$ is finite;}
        			\item{there exists a finite dimensional subspace $W$ of $L^2$ having an orthonormal basis $\varphi^1,\ldots,\varphi^k$ consisting of eigenfunctions of $T_V$ with eigenvalues $\lambda_1,\ldots,\lambda_k$ respectively. Each $\lambda_i$ is negative and $Q_V(\varphi,\varphi)\geq 0$ whenever $\varphi\in C^{\infty}_c\cap W^{\perp}$. In this case $k$ equals the Morse index of $Q_V$.}	
        		\end{enumerate}		
        	\end{prop}
        	
        	
        	\begin{rmk}\label{rmk:Vbounded} If $V$ is assumed to be bounded, then the basis provided in Proposition \ref{pro:l2spec} actually consists of elements of finite Dirichlet energy, i.e. functions belonging to $W^{1,2}$, the closure of $C^{\infty}_c$ with respect to the Sobolev norm determined by
        		\begin{equation}\label{eq:SobNorm}
        		\|\varphi\|^2_{W^{1,2}}:= \|\varphi\|^2_{L^2}+\|\nabla \varphi\|^2_{L^2}.
        		\end{equation}	
        	\end{rmk}

        	Based on this fact and in view of the geometric applications we wish to present, we \emph{assume} from now onwards that the function	$V\in C^{\infty}_{\mathcal{E}}$ is uniformly bounded.

        	It is convenient to introduce the relevant notion of (point) spectrum in this setting.
        	
        	\begin{definition}\label{def:spec} In the setting above, we define
        		\[
        		spec (Q_V):=\left\{\text{critical values of the map} \ \varphi\in W^{1,2}\setminus \left\{0\right\} \mapsto \frac{Q_V(\varphi,\varphi)}{\|\varphi\|^2_{L^2}}\right\};
        		\]
        		\[
        		spec_{\mathcal{E}} (Q_V):=\left\{\text{critical values of the map} \ \varphi\in (W^{1,2})_{\mathcal{E}}\setminus \left\{0\right\} \mapsto \frac{Q_V(\varphi,\varphi)}{\|\varphi\|^2_{L^2}}\right\};
        		\]
        		\[
        		spec_{\mathcal{O}} (Q_V):=\left\{\text{critical values of the map} \ \varphi\in (W^{1,2})_{\mathcal{O}}\setminus \left\{0\right\} \mapsto \frac{Q_V(\varphi,\varphi)}{\|\varphi\|^2_{L^2}}\right\}.
        		\]
        	\end{definition}
        To avoid ambiguities: $\lambda\in spec(Q_V)$ if and only if we can find an associated critical point $\varphi \in W^{1,2}\setminus\left\{0\right\}$, which satisfies
        	\begin{equation}\label{eq:weak}
        	\int_{\check{\Sigma}}(\nabla \varphi \cdot \nabla\zeta -V\varphi\zeta)\,d\h^n = \lambda \int_{\check{\Sigma}}\varphi\zeta\,d\h^n, \ \ \ \forall \zeta\in C^{\infty}_c,
        	\end{equation}
        	whence it is standard to get that $\varphi$ is actually smooth and solves, in a classical sense, the eigenvalue equation
        	\[
        	\Delta_g \varphi+V\varphi+\lambda\varphi=0
        	\]
        	which also implies that $\Delta_g\varphi \in L^2$. 
        	
        	Similar arguments can be applied to the other two cases as well, with an important caveat. By definition, we have that $\lambda\in spec_{\mathcal{E}} (Q_V)$ (respectively $\lambda\in spec_{\mathcal{O}} (Q_V)$) if equation \eqref{eq:weak} is satisfied for every \textsl{even} (respectively \textsl{odd}) test function $\zeta\in C^{\infty}_c$. However, we also have by symmetry arguments (cf. equation \eqref{eq:mixedprod}) that the same equation is satisfied for every \textsl{odd} (respectively \textsl{even}) test function, and hence for any $\zeta\in C^{\infty}_c$ (in either of the two cases). Therefore, we shall actually conclude
        	\[
        	\lambda\in spec_{\mathcal{E}} (Q_V) \ \Leftrightarrow \ \exists \varphi\in  C^{\infty}_{\mathcal{E}}\setminus \left\{0\right\} \ \text{such that} \	\Delta_g \varphi+V\varphi+\lambda\varphi=0,
        	\]
        	and
        	\[
        	\lambda\in spec_{\mathcal{O}} (Q_V) \ \Leftrightarrow \ \exists \varphi\in  C^{\infty}_{\mathcal{O}}\setminus \left\{0\right\} \ \text{such that} \	\Delta_g \varphi+V\varphi+\lambda\varphi=0.
        	\]

        	\begin{lem}\label{lem:specadd}
        		In the setting above,
        		\[
        		spec (Q_V)=spec_{\mathcal{E}} (Q_V) \cup 	spec_{\mathcal{O}} (Q_V).
        		\]
        		Moreover, under the assumption that the operator $T_V$ has finite Morse index on $(\check{\Sigma},g)$, we have
        		\[
        		Ind (Q_V) = Ind_{\mathcal{E}} (Q_V)+Ind_{\mathcal{O}} (Q_V).
        		\]
        	\end{lem}

        	\begin{proof}
        		For the first claim, notice that the inclusion $\supseteq$ is obvious so let us discuss the other one. Let then $\lambda \in spec (Q_V)$ and write $\varphi=\varphi_{\mathcal{E}}+\varphi_{\mathcal{O}}$ (the $L^2$ decomposition of an associated eigenfunction into even and odd parts), thus by linearity
        		\[
        		0=\Delta_g \varphi+V\varphi+\lambda\varphi= \underbrace{(\Delta_g \varphi_{\mathcal{E}}+V\varphi_{\mathcal{E}}+\lambda\varphi_{\mathcal{E}})}_{\in L^2_{\mathcal{E}} } + \underbrace{(\Delta_g \varphi_{\mathcal{O}}+V\varphi_{\mathcal{O}}+\lambda\varphi_{\mathcal{O}})}_{\in L^2_{\mathcal{O}}}
        		\]
        		hence each of the two summands must vanish and thus (since either $\varphi_{\mathcal{E}}\neq 0$ or $\varphi_{\mathcal{O}}\neq 0$) we get $\lambda\in spec_{\mathcal{E}} (Q_V)$ or $\lambda\in spec_{\mathcal{O}} (Q_V)$ as it was claimed.

        		For the second part, recall that by Proposition \ref{pro:l2spec} there exist (under the assumption that $T_V$ has finite Morse index) finitely many (say $k$) eigenfunctions $\varphi^1,\ldots,\varphi^k \in L^2$ that correspond to the negative eigenvalues $\lambda_1,\ldots,\lambda_k$ (where it is understood that each eigenvalue can be repeated if it comes with multiplicity). 
        		Following the argument we just presented, replace each $\varphi^j$ by means of the couple $\varphi^j_{\mathcal{E}}, \varphi^j_{\mathcal{O}}$ and set
        		\[
        		V=\text{span}_{\R}\left\{\varphi^1,\ldots,\varphi^k \right\}, \ \ \widetilde{V}=\text{span}_{\R}\left\{\varphi^1_{\mathcal{E}}, \varphi^1_{\mathcal{O}},\ldots, \varphi^k_{\mathcal{E}}, \varphi^k_{\mathcal{O}} \right\}.
        		\] 
        		Now, since $\varphi^j=\varphi^j_{\mathcal{E}}+\varphi^j_{\mathcal{O}}$ for any $j\in\left\{1,\ldots, k\right\}$ we have that
        		\[
        		\text{dim}_{\R}(\widetilde{V})\geq k
        		\]
        		and thus there are at least $k$ functions in the collection $\left\{\varphi^1_{\mathcal{E}}, \varphi^1_{\mathcal{O}},\ldots, \varphi^k_{\mathcal{E}}, \varphi^k_{\mathcal{O}} \right\}$ which are not zero, and eigenfunctions either for $spec_{\mathcal{E}} (Q_V)$  or for $spec_{\mathcal{O}} (Q_V)$, with negative eigenvalues. Thereby, the inequality one gets, combined together with Lemma \ref{lem:ineq}, allows to complete the proof.
        	\end{proof}

        	\subsection{The Morse index of half-bubbles}\label{subs:morsehalf}
        	
        	The discussion presented in the previous section directly applies, as a special case, to the study of the Jacobi operator
        	that is associated to the symmetrized complete minimal hypersurface that is obtained by reflecting a half-bubble. 
        	
        	Given a half-bubble $\Sigma$ we define, in analogy with Definition \ref{def:index} above, its Morse index considering the Jacobi form
\[        	
Q_{|A|^2}(u,u):=\int_{\Sigma}(|\nabla u|^2-|A|^2u^2)\,d\h^n.
\] 
More precisely, we give the following:
        	
        	\begin{definition} In the setting above, we define
        		$index (\Sigma)$ to be the largest dimension of a linear subspace of $C^{\infty}_c(\Sigma)$ where $Q_{|A|^2}$ is negative definite. 
        	\end{definition}	
        	\noindent Notice that here the support of $\varphi\in C^{\infty}_c(\Sigma)$ can intersect $\partial\Sigma$ so that we are imposing no condition along $\partial\Sigma$, and $index (\Sigma)$ is the (standard) Morse index of $\Sigma$ as a free boundary minimal hypersurface.
        	
        	\begin{lem}\label{lem:II1}
        A two-dimensional half-bubble has finite Morse index and this equals $Ind_{\mathcal{E}}(Q_{|\check{A}|^2})$ where $\check{\Sigma}$ is the double of $\Sigma$.
         The same conclusion holds true in all dimensions for any half-bubble $\Sigma$ having Euclidean volume growth, meaning that there exists a constant $C=C(\Sigma)$ such that $\h^n(\Sigma\cap B_R(0))\leq C R^n$ for all $R\geq 0$.
        	\end{lem}
        
           \begin{proof}
        	The argument for the first part ($n=2$) goes as follows: by definition $\Sigma$ has finite total curvature, so its double $\check{\Sigma}$ will also have finite total curvature, hence finite Morse index (again by \cite{FC85}); thus Lemma \ref{lem:ineq} implies that the corresponding \textsl{even} Morse index (that is to say: the Morse index on even test functions, as per Definition \ref{def:index}) is also finite. Now, going back to the definitions it is clear that by restriction $index (\Sigma)\geq Ind_{\mathcal{E}}(Q_{|\check{A}|^2})$.
        	
        	For the converse inequality, we argue by contradiction as follows. If it were $index (\Sigma)> Ind_{\mathcal{E}}(Q_{|\check{A}|^2})$, we could consider a basis $\left\{\varphi^1,\ldots, \varphi^I\right\}$ for a subspace of $C^{\infty}_c(\Sigma)$, having dimension equal to $I:=Ind_{\mathcal{E}}(Q_{|\check{A}|^2})+1$, where $Q_{|A|^2}$ is negative definite. 
        			If we extend each function by even symmetry, we obtain a family of even functions $\left\{\check{\varphi}^1,\ldots \check{\varphi}^I\right\}$ where the Jacobi form of $\check{\Sigma}$ is negative definite; each function is smooth away from $\partial\Sigma$  and we can consider for each $i\in\left\{1,\ldots, I\right\}$ an approximating sequence of smooth functions $\left\{\check{\varphi}^i_k\right\}$ that also have compact support, and such that $\check{\varphi}^i_k\to\check{\varphi}^i$ uniformly as one lets $k\to\infty$.
        	 		 We claim that for any $k$ large enough the family $\left\{\check{\varphi}^1_k,\ldots \check{\varphi}^I_k\right\}$ is linearly independent, which would then violate the definition of $Ind_{\mathcal{E}}(Q_{|\check{A}|^2})$ and conclude the proof.
        			 If the claim were false, we could write down (for every $k\in\mathbb{N}$) a linear equation of the form
        		 \[
        		 \sum_i a^i_k \check{\varphi}^i_k=0, \ \text{where} \ \sum_i (a^i_k)^2=1,
        		 \]
        		whence passing to the limit for $k\to\infty$, we end up finding a (non-trivial) linear relation involving 
        	$\left\{\varphi^1,\ldots, \varphi^I\right\}$, that is impossible since this family of functions was chosen to be a basis.
        	 In higher dimensions, one can follow the same argument modulo invoking the main theorem of J. Tysk in \cite{T89}.
        	\end{proof}

        	We can instead consider the Morse index of $\Sigma$ with Dirichlet boundary conditions:
        	
        	\begin{definition} In the setting above, we define
        		$index_{\bullet} (\Sigma)$ to be the largest dimension of a linear subspace of $C^{\infty}_c(\mathring{\Sigma})$ where $Q_{|A|^2}$ is negative definite.
        	\end{definition}
        	
        	Here $\mathring{\Sigma}:=\Sigma\setminus\partial\Sigma$ is the interior of $\Sigma$, and so this is the standard notion of Morse index for minimal surfaces with respect to variations that fix the boundary. Following the same conceptual scheme as above, we have the following ancillary result:
        	
        	\begin{lem}\label{lem:II2}
        		A two-dimensional half-bubble has finite Morse index \emph{with Dirichlet boundary conditions} and this equals $Ind_{\mathcal{O}}(Q_{|\check{A}|^2})$ where $\check{\Sigma}$ is the double of $\Sigma$.
        		The same conclusion holds true in all dimensions for any half-bubble $\Sigma$ having Euclidean volume growth, meaning that there exists a constant $C=C(\Sigma)$ such that $\h^n(\Sigma\cap B_R(0))\leq C R^n$ for all $R\geq 0$.
        	\end{lem}	
        	
        	\begin{proof}
        	Arguing as for the lemma above, let us focus on $n=2$. When $\Sigma$ is a two-dimensional half-bubble, the finiteness of $index_{\bullet} (\Sigma)\leq index(\Sigma)$ follows straight from Lemma \ref{lem:II1}. That being said, it is clear that we also have the inequality
        	$index_{\bullet} (\Sigma)\leq Ind_{\mathcal{O}}(Q_{|\check{A}|^2})$ by odd extension. The converse inequality
        	relies instead on the fact that odd functions must vanish along $\partial\Sigma=Fix(\tau)$ so that one can consider a linearly independent set of $Ind_{\mathcal{O}}(Q_{|\check{A}|^2})$ elements spanning a subspace on which $Q_{|\check{A}|^2}<0$, restrict each such function to $\Sigma$ and then truncate using a compactly supported cutoff function. We omit the standard details.
        		\end{proof}
        	
        Based on Lemma \ref{lem:II1} and \ref{lem:II2} one can rephrase the inequality $index_{\bullet} (\Sigma)\leq index(\Sigma)$ as
        	\begin{equation}\label{eq:EindexvsOindex}
        	Ind_{\mathcal{E}}(Q_{|\check{A}|^2})\geq  Ind_{\mathcal{O}}(Q_{|\check{A}|^2}),
        	\end{equation}
        	hence, using the second part of Lemma \ref{lem:specadd}, we get the estimate
        	\begin{equation}\label{eq:fundineq}
        	2 Ind_{\mathcal{O}}(Q_{|\check{A}|^2})\leq Ind(Q_{|\check{A}|^2})\leq 2 Ind_{\mathcal{E}}(Q_{|\check{A}|^2}).
        	\end{equation}
        	From there, we can derive some interesting characterizations for the cases when the Morse index of $\check{\Sigma}$ (which is, in our notation, $	Ind(Q_{|\check{A}|^2})$) equals 0, 1 or 2:
        	
        	\begin{cor}\label{cor:lowindex}
        		In the setting above, we have 
        		\begin{align*}
        		&\text{(1)}& 		Ind(Q_{|\check{A}|^2})=0 \ \Leftrightarrow Ind_{\mathcal{E}}(Q_{|\check{A}|^2})=0 \ \text{and} \ Ind_{\mathcal{O}}(Q_{|\check{A}|^2})=0; \\
        		&\text{(2)}& 		Ind(Q_{|\check{A}|^2})=1 \ \Leftrightarrow Ind_{\mathcal{E}}(Q_{|\check{A}|^2})=1 \ \text{and} \ Ind_{\mathcal{O}}(Q_{|\check{A}|^2})=0;\\
        		&\text{(3)}&		Ind(Q_{|\check{A}|^2})=2 \ \Leftrightarrow 
        				\begin{cases}Ind_{\mathcal{E}}(Q_{|\check{A}|^2})=2 \ \text{and} \ Ind_{\mathcal{O}}(Q_{|\check{A}|^2})=0 \ \ \textbf{or} \\ 
        				 Ind_{\mathcal{E}}(Q_{|\check{A}|^2})=1 \ \text{and} \ Ind_{\mathcal{O}}(Q_{|\check{A}|^2})=1.
        				\end{cases}
        				\end{align*}
        	\end{cor}
        	
        	So far our discussion is applicable to hypersurfaces of dimension $n\geq 2$, and we shall now specify them to the special case $n=2$, where some interesting consequences can be drawn.
        	Indeed, we can turn classification results for complete minimal surfaces in $\R^{3}$ into classification results for free boundary minimal hypersurfaces contained in a half-space. The first one is a Bernstein-type theorem for half-bubbles:
        	
        	\begin{cor}\label{cor:stable}
        		A two-dimensional, stable half-bubble is isometric to a half-plane.	
        	\end{cor}	
        	
        	\begin{proof}
        		It suffices to combine part (1) of Corollary \ref{cor:lowindex}, inequality \eqref{eq:EindexvsOindex} and the characterization of stable minimal surfaces in $\R^{3}$ (see \cite{dCP79,FCS80,Pog81}).
        	\end{proof}
        	
        	\begin{rmk}\label{rmk:higher-stable}
        	In fact, the stability estimates by Schoen-Simon allow to obtain a higher-dimensional counterpart of the previous result: \textsl{Let $\Sigma^n\subset\R^{n+1}$, $2\leq n\leq 6$ be a half-bubble. If $\Sigma$ is stable and has Euclidean volume growth, then $\Sigma$ is a hyperplane.}
        	\end{rmk}

        	Similarly, we get a simple result in the index one case:	
        	
        	\begin{cor}\label{cor:indexone}
        		A two-dimensional, index one half-bubble is isometric to a half-catenoid.	
        	\end{cor}	
        	
        	\begin{proof}
        		It suffices to combine inequality \eqref{eq:EindexvsOindex}, parts (2) and (3) of Corollary \ref{cor:lowindex}, the characterization of index one minimal surfaces in $\R^3$, see \cite{LR89}, and the non-existence result for index two complete minimal surfaces in $\R^3$, see \cite{CM14}.
        	\end{proof}
        	
        	In fact, inspired by recent work of Chodosh-Maximo \cite{CM14} we may wonder whether there exist two-dimensional half-bubbles whose Morse index equals two.
        	Obviously, a negative result would follow by proving that a half-bubble of Morse index 2 has vanishing Morse index with Dirichlet boundary conditions. Yet, this does not directly follow from the results above. 	
        	
        	\subsection{Geometry of half-bubbles}\label{subs:geomhalf}
        	
        	Let us discuss now some basic facts about the geometry and topology of half-bubbles. First of all, keeping in mind the quantization identity presented in Theorem \ref{thm:quant}, we wish to express the total curvature of a half-bubble in terms of its topology.
        	Clearly, in ambient dimension three ($n=2$) we can apply the Gauss equations to get
        	\begin{equation}\label{eq:half2nd}
        	2\mathcal{A}(\Sigma)=\mathcal{A}(\check{\Sigma})=-2\int_{\check{\Sigma}} K_{\check{\Sigma}}\,d\h^2
        	\end{equation}
        	while by Gauss-Bonnet (cf. Jorge-Meeks \cite{JM83}), denoted by $\gamma(\check{\Sigma})$ the genus of $\check{\Sigma}$ and by $b(\check{\Sigma})$ the number of its ends, we have
        	\begin{equation}\label{eq:GBcomplete}
        	\int_{\check{\Sigma}} K_{\check{\Sigma}}\,d\h^2=2\pi (\chi(\check{\Sigma})-b(\check{\Sigma}))=2\pi(2-2\gamma(\check{\Sigma})-2b(\check{\Sigma})) .
        	\end{equation}
        	It is then important to relate such data to the topological data of the half-bubble $\Sigma$. There are two easy examples of half-bubbles to be kept in mind throughout the following discussion, which we will prove in Lemma \ref{lem:euler} to be, roughly speaking, prototypical.
        	
        	\begin{definition}\label{def:horcurvercut} In the flat Euclidean space $\R^3$ we introduce the following terminology:
        		\begin{itemize} \item{the half-catenoid given by
        			\[
        			\left\{(x^1, x^2, x^3)\in \R^3 : \ (x^3)^2=\cosh ((x^1)^2+(x^2)^2), \ \ x^3\geq 0 \right\}
        			\]
        			will be called \textsl{horizontally cut half-catenoid};}	
       \item{the half-catenoid given by
        			\[
        		\left\{	(x^1, x^2, x^3)\in \R^3 : \ (x^3)^2=\cosh ((x^1)^2+(x^2)^2), \ \ x^1\geq 0 \right\}	
        			\]
        			will be called \textsl{vertically cut half-catenoid}.}
        			\end{itemize}
        	\end{definition}
        	
        	 Before proceeding further, it is helpful to recall some information about the asymptotic structure of complete minimal surfaces in $\R^3$.
        	
        	\begin{rmk}\label{rem:reginf}
        		Let $S$ be a complete, embedded minimal surface of $\R^3$. Then:
        		\begin{itemize}
        			\item{$S$ has finite total curvature if and only if it has finite Morse index (cf. \cite{FC85});}
        			\item{if $S$ satisfies either of these two equivalent assumptions then it is regular at infinity (cf. \cite{Sc83}), namely it can be decomposed, outside any sufficiently large compact set, into a finite number of connected components (named \textsl{ends}) and each such end can be described as a graph (over a plane in $\R^3$) of a defining function having an expansion of the form
        				\[
        				u(x')=a\log|x'|+b+\frac{c_1 x^1}{|x|^2}+\frac{c_2 x^2}{|x|^2}+O(|x'|^{-2}), \ |x'|\to\infty
        				\]
        				in suitable coordinates $x=(x^1,x^2,x^3)$, where $x'=(x^1,x^2)$.
        			}
        		\end{itemize}
        		Notice that if $\Sigma$ is a half-bubble then both facts apply to its double $\check{\Sigma}$ and hence, a posteriori, one can talk about the ends of a half-bubble and this notion is well-defined. Furthermore, each end of a half-bubble will have a precise asymptotic description of the type above.
        	\end{rmk}
        	
        		\begin{rmk}
        			We can somewhat strengthen Lemma \ref{lem:II1} and extend to half-bubbles the aforementioned equivalence result by Fischer-Colbrie:
        			\end{rmk}
        		\begin{prop}\label{pro:equiv}
        			A two-dimensional half-bubble has finite index if and only if it has finite total curvature.	The same conclusion holds true in all dimensions provided Euclidean volume growth is assumed.
        		\end{prop}
        		One of the two implications has been proven above in Lemma \ref{lem:II1}, while the other follows by combining Lemma \ref{lem:specadd} and the inequalities in \eqref{eq:fundineq}.
        	
        	\
        	
        	Let us now proceed in our discussion, relating all half-bubbles to one of the two models given in Definition \ref{def:horcurvercut}.
        	
        	\begin{lem}\label{lem:euler}
        		Let $\Sigma$ be a half-bubble and let $\check{\Sigma}$ denote its double in $\R^3$. Then one of the following two alternative cases occurs:
        		\begin{enumerate}
        			\item{the number of ends of $\check{\Sigma}$ equals double the number of ends of $\Sigma$, and their Euler characteristics are related by the equation
        				\[
        				\chi(\check{\Sigma})=2\chi(\Sigma);
        				\]	
        			}
        			\item{the number of ends of $\check{\Sigma}$ equals the number of ends of $\Sigma$,  and their Euler characteristics are related by the equation
        				\[
        				\chi(\check{\Sigma})=2\chi(\Sigma)-b(\Sigma).
        				\]
        			}	
        			\end{enumerate}
        			In either case, one has that $\chi(\check{\Sigma})-b(\check{\Sigma})=2(\chi(\Sigma)-b(\Sigma))$.
        	\end{lem}	
        	
        	Clearly, both cases do occur: (1) is exemplified by the horizontally cut half-catenoid, while (2) is exemplified by the vertically cut half-catenoid. 
        	
        	\begin{proof}
        		Based on Remark \ref{rem:reginf} applied to the symmetrized bubble $\check{\Sigma}$ we have that either each of its ends is contained in one of the half-spaces $\Pi_1$ and $\Pi_2$ (provided we remove from $\check{\Sigma}$ a sufficiently large ball centered at the origin) or instead all of them intersect $\Pi$. By symmetry, it is clear that in the first case the number of ends of $\check{\Sigma}$ equals double the number of ends of $\Sigma$ and we can justify the equation for the Euler characteristic using the Gauss-Bonnet theorem: since $\Sigma$ is assumed to have finite total curvature we can write for $\Sigma_r:=\Sigma \cap \left\{|x|\leq r\right\}$
        		\[
        		\int_{\Sigma}K_{\Sigma}\,d\h^2=\lim_{r\to\infty}\int_{\Sigma_r} K_{\Sigma}\,d\h^2
        		\]
        		and we can compute the limit on the right-hand side along a sequence of generic radii $\left\{r_i\right\}$, so that the intersection $\Sigma\cap \left\{|x|=r_i\right\}$ is transverse and thus consisting of finitely many, smooth closed curves, hence
        		\[
        		\int_{\Sigma_{r_i}} K_{\Sigma}\,d\h^2=2\pi \chi(\Sigma)-2\pi b(\Sigma)+o(1)
        		\]
        		(where we have used the fact that the geodesic curvature of $\partial\Sigma\subset\Sigma$ is zero at all points).
        		Thus the conclusion comes straight from combining this equation with \eqref{eq:half2nd} and \eqref{eq:GBcomplete}. This is the scenario corresponding to case (1).
        		
        		If instead all ends of $\check{\Sigma}$ intersect the plane $\Pi$ then we are in case (2) and the claim concerning $\chi(\check{\Sigma})$ comes, once again, by an elementary computation: this time we have, along a sequence of generic radii 
        		\[
        		\int_{\Sigma_{r_i}} K_{\Sigma}\,d\h^2=2\pi \chi(\Sigma)-\pi b(\Sigma)-\pi b(\Sigma)+o(1)
        		\]
        		by keeping the contributions of both the geodesic curvature and of the exterior angles (at the non-smooth points of $\Sigma_r$) into account. In that respect, notice that  there are exactly $2b(\Sigma)$ exterior angles, each of them being $\pi/2+o(1)$ as $i\to\infty$. Thereby, the conclusion follows.
        	\end{proof}

        	We can extend to half-bubbles some further classification results that are well-known for complete minimal surfaces in $\R^3$. This discussion parallels the one presented in the previous subsection, which was based on index-theoretic criteria instead.

        	\begin{cor}\label{cor:oneend}
        		A two-dimensional half-bubble with one end is isometric to either a half-plane or to a horizontally cut half-catenoid.
        	\end{cor}	
        	
        	\begin{proof}
        		Denoted by $\Sigma$ the half-bubble in question and $\check{\Sigma}$ its double, Lemma \ref{lem:euler} implies that \textsl{either} $\Sigma$ and $\check{\Sigma}$ have the same number of ends (one) \textsl{or} $\check{\Sigma}$ will have two ends. By classical results of Schoen (see \cite{Sc83}) the first case happens if and only if $\check{\Sigma}$ is a plane, and the second if $\check{\Sigma}$ is a catenoid. The conclusion in the former case is straightforward, so let us discuss the latter. The only way $\Sigma$ has exactly one end (which we are assuming) is that it is cut with a plane that is parallel to its two ends. But then the only way such a cut gives rise to a free boundary minimal surface is when the plane in question passes through the center of the catenoid, which proves the claim.
        	\end{proof}	
        	
        	\begin{cor}\label{cor:genuszero}
        		A two-dimensional half-bubble whose Euler characteristic equals one is isometric to either a half-plane or to a vertically cut half-catenoid.
        	\end{cor}	
        	
        	\begin{proof}
        		Like we did above, we need to consider two cases and apply the equations provided by Lemma \ref{lem:euler} for the Euler characteristics. If we are in case (1), the symmetrized bubble $\check{\Sigma}$ satisfies $\chi(\check{\Sigma})=2$ and this is not possible because the standard Jorge-Meeks formula for the Euler characteristic (cf. the second equality in \eqref{eq:GBcomplete}) forces $\chi(\check{\Sigma})\leq 1$. 
        		Instead, in case (2) we would get that the bubble $\check{\Sigma}$ has genus zero, which allows to invoke \cite{LR91}: $\check{\Sigma}$ is either a plane or a catenoid. At this stage, the conclusion comes at once by considering all possible planes of symmetry of a catenoid.
        	\end{proof}

        	 	\section{Bubbling analysis}\label{sec:bubbling}

        	 	In this section, we prove Theorem \ref{thm:bubbles} and in Section \ref{sec:neck} we complete the proof of Theorem \ref{thm:quant}. Before we begin either of these proofs notice that given $y\in \mathcal{Y}\setminus \partial N$ the bubbling and neck analysis in \cite{BS17} goes through precisely as before since it is of a local nature. In particular the conclusions of Theorems \ref{thm:quant} and \ref{thm:bubbles} hold over smooth domains $U\cemb N\setminus\partial N$ such that $\mathcal{Y}\cap \partial U =\emptyset$ and $y\in \mathcal{Y}\setminus \partial N$ respectively, which implies that the proofs of both of these theorems is complete in the case that $\mathcal{Y}\cap \partial N$ is empty. We need only focus on the setting $y\in \mathcal{Y}\cap \partial N$ from now on. 
        	 	
        	 	Our main result here is a blow-up theorem (Theorem \ref{thm:blowup} below) which will detect a non-trivial bubble or half-bubble in all regions of coalescing index at $\partial N$. 
        	 	This result is based on a localized version of the compactness result for free boundary minimal hypersurfaces from \cite{ACS17}. Once we have proved Theorem \ref{thm:blowup}, we can then follow a scheme related to the bubbling analysis for closed minimal hypersurfaces developed in Section 3 of \cite{BS17} to construct various point-scale sequences which will detect \emph{all} the non-trivial bubbles and half-bubbles that develop at points of curvature concentration on $\partial N$, yielding Theorem \ref{thm:bubbles}. 
        	 	
        	 	In order to obtain the quantization result in Theorem \ref{thm:quant} from our proof of Theorem \ref{thm:bubbles}, it remains to show that no curvature is lost in \emph{annular neck regions} between the bubble scales. This last step will be carried out in Section \ref{sec:neck}.\\

        	 	In order to state our results in a unified manner, we will use the following notation and conventions throughout. Consistently with Section \ref{sec:prelim} denote by $\Pi_1(a):=\{x^1\geq a\}$ a closed half-space of $\R^{n+1}$ \emph{containing the origin} and by $\Pi(a)$ its boundary (though we drop the dependence on $a$ when it is irrelevant).  We will consider 
        	 	a relatively open sub-domain $C^{\R^{n+1}}_\varrho(a)\subset \Pi_1$,  
        	 	defined by $C^{\R^{n+1}}_\varrho(a)=[a,a+\varrho)\times B^{\R^n}_\varrho (0)$. We will allow $\varrho=\infty$ in the sequel at which point the domain is simply $\Pi_1(a)$.  
        	 	Now equip $C^{\R^{n+1}}_\varrho$ with a smooth metric $g$ such that it is Euclidean at $(a,0,\dots,0)$ and for all $z\in C^{\R^{n+1}}_\varrho\cap \Pi$, $\partial_{x^1}(z)$ is the inward unit normal to $\Pi$ with respect to $g$. We say that $(C^{\R^{n+1}}_\varrho,g)$ is an \emph{adapted} domain and in the sequel we will allow the parameters $\varrho$ and $g$ to vary under these constraints, and $a$ will largely remain fixed.   

        	 	Such domains $(C^{\R^{n+1}}_\varrho,g)$ can be used to analyze local properties of free boundary minimal hypersurfaces in arbitrary $N$ without loss of generality via the use of
        	 	a \emph{Fermi-coordinate} neighborhood at the boundary of $N$: 
        	 	Let $-\nu$ be the inward pointing normal to $\partial N$ and $\gamma_1(N)>0$ be so that 
        	 	\begin{eqnarray*}
        	 		& F:[0,\gamma)\times \partial N \to N&\\
        	 		&(t,z)\mapsto Exp^N_z (-t\nu)&
        	 	\end{eqnarray*} 
        	 	is injective for $t\in [0,\gamma_1)$. 
        	 	
        	 	Consider a normal coordinate neighborhood of $q$ in $\partial N$, $\{x_i\}_{i=2}^{n+1}$. For ease of notation, we will now choose $\gamma_0(N) >0$ to be smaller than both $\gamma_1$ defined above, and the injectivity radius of $\partial N$ so that, first of all, these boundary normal coordinates are defined on $B^{\partial N}_\gamma (q)$ for any $q\in \partial N$ and $\gamma \leq \gamma_0$. Second of all, for all $q\in \partial N$ and $\gamma\leq \gamma_0$, we may now extend these to a Fermi-coordinate neighborhood $C_{\gamma}(q) \cong [a,a+\gamma)\times B^{\R^n}_\gamma (0)$ via the exponential map giving $x^{1} = a+ Exp^N_{(0,x^2,\dots,x^{n+1})} (-t\nu)$ for $0\leq t<\gamma$. 
        	 	
        	 	Notice that the metric $g$ in these coordinates coincides with the Euclidean metric at $(a,0,\dots,0)$. Moreover $\partial_{x^1}(z) = -\nu(z)$, in particular $g_{1j}(z)=0$, for all $z\in \partial N = \{x^1=a\}$ and $j\geq 2$. In particular $C_\gamma(q)\cong C^{\R^{n+1}}_\gamma(a)$ is an adapted domain.  \\
        	 	
        	 	For $V\subset (C^{\R^{n+1}}_\varrho,g)$ consider the set $\mathfrak{M}^{V}$ (which shall depend on both $C^{\R^{n+1}}_\varrho$ and the background metric $g$) of smooth, connected and properly embedded minimal hypersurfaces $P\subset V$, furthermore requiring that if $P\cap \Pi \neq \emptyset$ then $P$ is free boundary with respect to $\Pi$. At the level of regularity, we always tacitly assume $V$ to be the intersection of a \emph{smooth} domain in $\R^{n+1}$ with $\Pi_1$. Any variational properties of $P$ are computed with respect to compactly supported variations in $V$ -- i.e. free boundary variations on $V\cap \Pi$ as well as Dirichlet conditions on $\partial V\setminus \Pi$. Following our introduction, we define
        	 	\begin{equation*}
        	 	\mathfrak{M}^{V}_p(\Lambda,\mu) = \big\{ P \in \mathfrak{M}^{V} :  \forall x\in C^{\R^{n+1}}_\varrho, R>0,  \ \h^n(P\cap B^{\R^{n+1}}_R(x))\leq \Lambda R^n, \ \text{ and }\lambda_p(P) \geq -\mu \big\}.
        	 	\end{equation*}

        	 	We recall from \cite[Theorem 29]{ACS17} that if $2\leq n\leq 6$ and $\{M_k\}\subset \mathfrak{M}_p(\Lambda, \mu) \subset \mathfrak{M}$ then there exists a smooth, connected, compact embedded minimal hypersurface $M\subset N$ meeting $\partial N$ orthogonally along $\partial M$, $m\in \mathbb{N}$ and a finite set $\mathcal{Y}\subset M$ with cardinality $|\mathcal{Y}|\leq p-1$  such that, up to subsequence, $M_k \to M$ locally smoothly and graphically on $M\setminus \mathcal{Y}$ with multiplicity $m$. To avoid ambiguities, let us remark that from now on we will always assume that  $\mathcal{Y}$ is the \emph{minimal} such set, so that in particular the convergence is never smooth about any $y\in\mathcal{Y}$.
        	 	
        	 	We will require the following local version of that compactness theorem.
        	 	
        	 	\begin{lem}\label{lemma:local}
        	 		Let $2\leq n \leq 6$ and $\{(C_{k}, g_k)\}$ be a sequence of adapted domains where we set $C_k=C^{\R^{n+1}}_{\varrho_k}(a)$ with $\{\varrho_k\}$ monotonically increasing, and choose $\eps >0$ small enough that $B^{\R^{n+1}}_\eps(0)\cemb C_1$. We will assume that $g_k\to g$ smoothly for some limit metric $g$ on any relatively open subset $V\cemb C_k$. Suppose we have a sequence $P_k\in \mathfrak{M}^{C_k}_p(\Lambda,\mu)$ (for some fixed constants $\Lambda, \mu\in \R_{>0}$ and a positive integer $p$ independent of $k$) such that $P_k\cap B^{\R^{n+1}}_{\eps}(0)\neq \emptyset$.
        	 		Then, up to subsequence:

        	 		\begin{enumerate}
        	 			\item For any relatively open $V$ such that $B_\eps^{\R^{n+1}}(0)\subset V\cemb C_k$ for some $k$, there exists a smooth, connected, and embedded minimal hypersurface $P\subset V$ where $P_k \to P$ locally smoothly and graphically, with multiplicity $m\in \N$, for all $x\in P\setminus \mathcal{Y}$ where $\mathcal{Y}\subset P$ is a finite set with cardinality $|\mathcal{Y}|\leq p-1$. 
        	 			
        	 			 If $\partial P\neq\emptyset$, then $P$ meets $\Pi$ orthogonally along $\partial P$ and if $P \in \mathfrak{M}^V$, then $P \in \mathfrak{M}^V_p(\Lambda,\mu)$.
        	 			 
        	 			  Finally $\mathcal{Y}\neq \emptyset$ if and only if there exists $\{y_k\}\subset V$ with $y_k\to y\in P$, and $r_k\to 0$ so that $index(P_k\cap B^{\R^{n+1}}_{r_k}(y_k))\geq 1$. In either case we say that $y\in \mathcal{Y}$. 
        	 			
        	 			
        	 			
        	 			\item Assuming that $\varrho_k \to \infty$, then there exists a limit $M$ which is a smooth, embedded  minimal hypersurface in $(\Pi_1,g)$.   
        	 		If we additionally assume $\liminf_{k\to\infty}\lambda_p(P_k)\geq 0$ then  $P \in \mathfrak{M}^{\Pi_1}$ implies $P \in \mathfrak{M}^{\Pi_1}(\Lambda,p-1)$.
        	 			
        	 			Further assuming $g=g_0$ is the Euclidean metric then $P$ is either a bubble or a half-bubble; if $P$ is not properly embedded, then $P\equiv \Pi$. 
        	 			
        	 			Finally if $\mathcal{Y}\neq \emptyset$ then $P$ must be a plane in $\Pi_1$ or a half-plane orthogonal to $\Pi$. 
        	 			
        	 		\end{enumerate}
        	 		
        	 	\end{lem}
        	 	
        	 	\begin{proof}
        	 		Part (1):        		 All the steps in the proof of \cite[Theorem 29]{ACS17} are local, hence can be adopted almost verbatim with only some cosmetic changes to conclude all but the final statement concerning the equivalent characterization of the condition $\mathcal{Y}\neq\emptyset$. To prove that assertion, let us first observe that since $P$ is smooth and $V$ is compact we can pick $r>0$ so that $P\cap B^{\R^{n+1}}_r(z)\cap V$ is \emph{strictly} stable for all $z\in P$. Thus, if $\mathcal{Y}=\emptyset$ then the convergence would be smooth everywhere, hence in particular for $k$ large enough each hypersurface $P_k$ would be stable in all balls of radius $r/2$ centered at points in $V$ (as defined above) by smooth convergence. Hence no sequence as in the statement can actually exist. Instead, for the converse implication one just needs to invoke the interior and free boundary curvature estimates of Schoen-Simon \cite{SS81} (cf. \cite[Theorem 19]{ACS17}).  
        	 		
        	 		\
        	 		
        	 		Part (2):
        	 		We can apply Part (1) to some compact exhaustion $\{V_{\ell}\}$ of $\Pi_1$ to conclude the first statement via a diagonal sequence argument.
        	 		
        	 		If $P$ is proper, i.e. $P \in \mathfrak{M}^{\Pi_1}$, then by Part (1) $P\in \mathfrak{M}^V_p(\Lambda, \mu)$ for \emph{all} $V \subset \Pi_1$. We now check that in this case in fact $index(P)\leq p-1$. Assuming the contrary guarantees the existence of a set $V$ with $\lambda_p(P\cap V)=\alpha<0$ and therefore, following the argument in \cite[p.~2599--2600]{ACS15}, we find that eventually $\lambda_p(P_k\cap V)\leq \frac{\alpha}{2}$, contradicting the assumption that $\liminf_{k\to\infty}\lambda_p(P_k)\geq 0$.  
        	 		
        	 		If $g=g_0$ is the Euclidean metric, then we distinguish two cases. If $P$ is proper, then exploiting the assumption of Euclidean volume growth, we conclude by Proposition \ref{pro:equiv} that it has finite total curvature. If on the other hand $P$ is not proper, i.e. an interior point of $P$ lies in $\Pi$, then by a maximum principle argument $P=\Pi$, and therefore $P$ trivially has finite total curvature. 
        	 		
        	 		Finally, if $\mathcal{Y}\neq \emptyset$ then we may as well assume that $P$ is properly embedded (if not then $P=\Pi$ and we are done). Since $P$ is properly embedded and $\Pi_1$ is simply connected we have that $P$ is two-sided. When we couple this with the lack of smooth convergence, we conclude that the convergence must be of multiplicity $m\geq 2$ (via Allard's interior and boundary regularity cf. \cite[Theorems 5, 17]{ACS17}) which allows for the construction of a positive Jacobi field over compact subsets $W\cemb \Pi_1$ (not necessarily vanishing at the boundary of $W$). In turn, this implies that $\lambda_1(P\cap W)\geq 0$ and the limit is stable over all such $W$ (here we use the standard argument to establish the positiveness of a first eigenfunction, noting that both an interior and boundary Hopf maximum principle will be required); hence $P$ is stable in $\Pi_1$. The conclusion follows by the usual classification of stable hypersurfaces with Euclidean volume growth when $P$ has no boundary \cite{SS81}, and Remark \ref{rmk:higher-stable} when $P$ has a free boundary on $\Pi$. 
        	 		\end{proof}

        	 	We can now establish the first main result of this section as a consequence of the above statement as well as the localized compactness theory developed in Section 2 of \cite{BS17} for the case of closed hypersurfaces. 
        	 	
        	 	 We have already seen in Lemma \ref{lemma:local} that bad points of convergence $y\in \mathcal{Y}$ correspond to coalescing regions of index, so these are the points that we analyze in detail here. We will denote by $B_r(p)$ a normal coordinate neighborhood of $p\in N$ and by $C_r(q)$ a Fermi-coordinate neighborhood of $q\in \partial N$.

        	 	\begin{thm}\label{thm:blowup}
        	 		Let $\{M_k\}\subset \mathfrak{M}_p(\Lambda, \mu) \subset \mathfrak{M}$ so that (up to subsequence) $M_k\to M$ for $M$ as in \cite[Theorem 29]{ACS17} and choose $\delta>0$ so that 
        	 		\begin{equation*}
        	 		2\delta<\min\left\{\inf_{ y_i\neq y_j\in \mathcal{Y}} dist_g(y_i,y_j),\,\,\, inj_N,\,\,\, \gamma_0\right\}.
        	 		\end{equation*} Assume the existence of sequences $\{\varrho_k\},\{r_k\}\subset \R_{>0}$ satisfying $r_k\to 0$, $\varrho_k\leq \delta$, $\varrho_k/r_k\to \infty$ and $\{p_k\in M_k\}$ with $p_k\to y\in \partial N$. We assume furthermore that $P_k$ is some connected component of $M_k\cap B_{\varrho_k}(p_k)$ and it satisfies
        	 		\begin{itemize}
        	 			\item $P_k\cap B_{r_k + (r_k)^2}(p_k)$ is unstable for all $k$,
        	 			\item $P_k\cap B_{r_k/2}(z)$ is stable for all $z\in P_k\cap B_{\varrho_k}(p_k)$. 
        	 		\end{itemize}
        	 		
        	 		After possibly passing to a further subsequence, we are in one of the following cases: \\
        	 		
        	 		\framebox[1.8cm]{Case I:} $\frac{dist_g(p_k,\partial N)}{r_k}\to\infty$. Choose normal coordinates for $N$ centered at $p_k$ on the geodesic balls $B_{\varrho_k}(p_k)$ and consider the rescaled hypersurfaces $\widetilde{P}_k \subset B^{\R^{n+1}}_{\varrho_k/r_k}(0)$ defined by
        	 		\begin{equation*}
        	 		\widetilde{P}_k  : = \frac{1}{r_k} \big(P_k\cap B_{\varrho_k}(p_k)- p_k\big).
        	 		\end{equation*}		 Then $\widetilde{P}_k$ must converge smoothly on any compact set with multiplicity one to a non-trivial bubble $\Sigma\subset \R^{n+1}$ with $\Sigma\cap B^{\R^{n+1}}_{3/2}(0)$ unstable. \\

        	 		\framebox[1.8cm]{Case II:} $\frac{dist_g(p_k,\partial N)}{r_k}\to c\geq 0$. Let $q_k\in\partial N$ be so that $dist_g(q_k,p_k)=dist_g(p_k,\partial N)$ and $s_k$ a point on the geodesic connecting $q_k$ to $p_k$ which is a distance $cr_k$ from $q_k$. Choose a Fermi-coordinate neighborhood (so that $a=-c$) $C_{\varrho_k}(q_k)$ and consider the rescaled hypersurfaces $\widetilde{P}_k \subset C^{\R^{n+1}}_{\varrho_k/r_k}(-c)$ defined by
        	 		\begin{equation*}
        	 		\widetilde{P}_k  : = \frac{1}{r_k} \big(P_k\cap C_{\varrho_k}(q_k)- s_k\big).
        	 		\end{equation*}   In this case $\widetilde{P}_k$ must converge smoothly on any compact set with multiplicity one to \emph{either a full or half-bubble} $\Sigma\subset \Pi_1(-c)$, non-trivial in both cases, with $\Sigma\cap B^{\R^{n+1}}_{3/2}(0)$ unstable. 
        	 		In the case that $n=2$, $\Sigma$ must be a non-trivial half-bubble. \\ 	
        	 		Moreover, in each of the above cases, we must have that $\lambda_1(P_k \cap (B_{2r_k}(p_k)))\to -\infty$. \\

        	 		Finally, suppose that there is another point-scale sequence $(\widehat{p}_k, \widehat{r}_k)$ with $\widehat{r}_k \geq r_k$ for all $k$ and such that
        	 		\begin{itemize}
        	 			\item $P_k\cap (B_{\widehat{r}_k+ (\widehat{r}_k)^2}(\widehat{p}_k)\setminus B_{2r_k}(p_k))$ is unstable for all $k$,
        	 			\item there exists some $C<\infty$ with $B_{\widehat{r}_k}(\widehat{p}_k)\subset B_{Cr_k}(p_k)$ for all $k$,
        	 		\end{itemize}
        	 		then we also have $\lambda_1(P_k\cap (B_{2\widehat{r}_k}(\widehat{p}_k)\setminus B_{2r_k}(p_k))) \to -\infty.$
        	 	\end{thm}
        	 	
        	 	\begin{rmk}\label{rmk:coord}
        	 		Suppose that we are in the setting of Case II, and consider $B_{\varrho_k}(p_k)\cap C_{\varrho_k}(q_k)$. Letting $\Phi_k : B_{\varrho_k}(p_k) \to B^{\R^{n+1}}_{\varrho_k/r_k}(0)$ and $\Psi_k : C_{\varrho_k}(q_k)\to C^{\R^{n+1}}_{\varrho_k/r_k}(-c)$ denote the blown-up normal and Fermi-coordinates respectively, one can easily check that 
        	 		$$\Phi_k\circ \Psi_k^{-1}:  \Psi_k(B_{\varrho_k}(p_k)\cap C_{\varrho_k}(q_k)) \to \Phi_k(B_{\varrho_k}(p_k)\cap C_{\varrho_k}(q_k))$$
        	 		converges locally smoothly on $\Pi_1$ to the identity as $k\to \infty$. Therefore the resulting blow-up and conclusions in Case II can actually be taken exactly as in Case I with respect to normal coordinates centered at $p_k$, or indeed as stated in the Theorem. 
        	 	\end{rmk}
        	 	
        	 	\begin{rmk}\label{rmk:equivballs}
        	 			Referring to the statement above, we further note that in these coordinates geodesic balls in $N$, $B_r(p)$, are directly comparable to Euclidean balls: for any $T>0$ there exists a sequence $\left\{\beta_k(T)\right\}$ with $\beta_k(T) \searrow 1$ so that for all $p\in B_{Tr_k}(s_k)$, if we denote $\widetilde{p}\in C^{\R^{n+1}}_{T}(-c)$ the point corresponding to $p$ in these rescaled coordinates then, by abusing notation,  
        	 			$$B^{\R^{n+1}}_{r\beta_k^{-1}}(\widetilde{p}) \subset B_r(p)\subset B^{\R^{n+1}}_{r \beta_k}(\widetilde{p}).$$
        	 		 This equivalence will be exploited along the course of the following proof when transferring certain variational properties (typically: eigenvalue bounds) back and forth between geodesic balls and coordinate balls. As it is well-known, the same equivalence holds true in geodesic normal coordinates as well.
        	 	\end{rmk}

        	 	\begin{proof}
        	 	
        	 		Case I: We can follow the arguments exactly as in \cite[Corollary 2.6]{BS17} to conclude all the statements of the theorem. \\

        	 		Case II: 
        	 		In these Fermi-coordinates, we are now in a position to apply Lemma \ref{lemma:local} and it remains to check that the limit $P$ is in fact a non-trivial (namely: it is not flat) and that the final statement of the theorem, concerning the point-scale sequence $(\widehat{p}_k,\widehat{r}_k)$, also holds.

        	 		Concerning the first assertion, observe that by Lemma \ref{lemma:local} it suffices to check that $P$ is not stable. 
        	 		We have $\mathcal{Y}=\emptyset$ since $\widetilde{P}_k$ is stable on all balls of radius $1/4$ (by virtue of our assumption), hence $\widetilde{P}_k$ converges smoothly and graphically on every compact subset of $\Pi_1$ to a connected minimal hypersurface $P$ and if $\partial P\neq\emptyset$ then $P$ meets $\Pi$ orthogonally. Furthermore either $P$ is properly embedded or $P=\Pi$.  
        	 		As the ambient space is simply connected, we can always conclude that $P$ is two-sided and that this convergence happens with multiplicity one.

        	 	Towards a contradiction, suppose that either $P=\Pi$ or
        	 	$$\lambda_1(P\cap B^{\R^{n+1}}_{3/2}(0))\geq 0$$ so that in particular $P$ is strictly stable on $B^{\R^{n+1}}_{4/3}(0)$. Notice that if $P=\Pi$ then the above is trivially true for variations in the whole of $\R^{n+1}$ (i.e. not just in $\Pi_1$). In either case, by smooth multiplicity one convergence we have that $\widetilde{P}_k$ is strictly stable on $B^{\R^{n+1}}_{4/3}(0)$ which implies (scaling back and using the second part of Remark \ref{rmk:coord}) that $P_k$ is stable on $B_{r_k + r_k^2}(p_k)$, a contradiction. In particular $P$ cannot be planar, and we have 
        	 	$$\lambda_1(P\cap B^{\R^{n+1}}_{3/2}(0))\leq -2\lambda_{\ast} < 0$$ for some $\lambda_{\ast}>0$, and hence for all $k$ sufficiently large
        	 	
        	 		$$\lambda_1(\widetilde{P}_k\cap B^{\R^{n+1}}_{3/2}(0))\leq -\lambda_{\ast}<0.$$ 
        	 		Hence a rescaling argument implies $$\lambda_1(P_k \cap (B_{2r_k}(p_k)))\to -\infty.$$ 
        	 		
        	 		Obviously, the fact that for $n=2$ we must obtain a half-bubble relies on the half-space theorem \cite{HM90} (i.e. there cannot be a non-trivial bubble contained in an open half-space.)
        	 		
        	 		Finally, the argument for the last part is similar to the above except we always choose normal coordinates to carry out our blow-ups using Remark \ref{rmk:coord}. At the $r_k$ scale we have that $B_{\widehat{r}_k}(\widehat{p}_k)$ is similar to the balls 
        	 		$$B^{\R^{n+1}}_{\beta_k^{\pm 1}\widetilde{r}_k}(\widetilde{p}_k)$$ with $1\leq \widetilde{r}_k \leq C$ and $\widetilde{p}_k\in B^{\R^{n+1}}_C(0)$, at which point we can conclude that $\widetilde{r}_k \to \widetilde{r}\in [1,C]$ and $\widetilde{p}_k \to \widetilde{p}\in B^{\R^{n+1}}_C(0)$.    
        	 		Assuming that $(B^{\R^{n+1}}_{2\widetilde{r}}(\widetilde{p})\setminus B^{\R^{n+1}}_2(0))\cap P$ is stable (or empty), we arrive at a contradiction in a similar fashion as above: first this implies that $B_{3\widetilde{r}/2}^{\R^{n+1}}(\widetilde{p})\setminus B_2^{\R^{n+1}}(0)\cap P$ is strictly stable (or empty). But this domain contains the blown-up initial domain on which we are assuming that $P_k$ is unstable. Thus this must be non-empty, and by the smooth multiplicity one convergence we would obtain that $(B_{3\widetilde{r}/2}^{\R^{n+1}}(\widetilde{p})\setminus B_2^{\R^{n+1}}(0))\cap \widetilde{P}_k$ is strictly stable, a contradiction. 
        	 		
        	 		Thus, we have proved that $$\lambda_1 ((B^{\R^{n+1}}_{3\widetilde{r}/2}(\widetilde{p})\setminus B^{\R^{n+1}}_2(0))\cap P)\leq -2\lambda_{\ast}<0 $$ as before, for some $\lambda_{\ast}>0$. A rescaling argument concludes the final statement of the Theorem, once again appealing to Remark \ref{rmk:equivballs}.         
        	 \end{proof}

        	 	The previous result can be regarded as a first step towards a proof of the Theorems \ref{thm:quant} and \ref{thm:bubbles}. Indeed, in the given setup, assume that $\delta$ is sufficiently small so that
        	 	\begin{equation*}
        	 	2\delta<\min\left\{\inf_{ y_i\neq y_j\in \mathcal{Y}} d_g(y_i,y_j),\,\,\, inj_N,\,\,\,\gamma_0\right\}.
        	 	\end{equation*}
        	 	
        	 	We know that the first part of Theorem \ref{thm:quant} (namely: the smooth graphical convergence with multiplicity $m$ away from the finite set $\mathcal{Y}$), holds true by \cite{ACS17} and therefore, we patently derive
        	 	\begin{equation*}
        	 	\int_{M_k\setminus \left(\bigcup_{y_i \in \mathcal{Y}} B_{\delta}(y_i)\right)} |A_k|^n d \h^n \to m \int_{M\setminus \left(\bigcup_{y_i \in \mathcal{Y}} B_{\delta}(y_i)\right)} |A|^n d \h^n.
        	 	\end{equation*}
        	 	Furthermore, as remarked at the beginning of this section, the full result of Theorem \ref{thm:bubbles} holds for any $y\in \mathcal{Y}\setminus\partial N$, and if $U\cemb N\setminus \partial N$ is a smooth domain with $\mathcal{Y}\cap \partial U=\emptyset$ then the bubbling analysis carried out in \cite{BS17} yields 
        	 	\begin{align*}
        	 	\mathcal{A}(M_k\cap U) &\to m\mathcal{A}(M\cap U) + \sum_{j=1}^{J_U} \mathcal{A}(\Sigma_j), \ \ \ (k\to\infty)
        	 	\end{align*}
        	 	where $J_U$ denotes the (non-trivial) full bubbles forming at point in $\mathcal{Y}\cap U$. 
        	 	
        	 	Putting these two threads together, we obtain in fact
        	 	\begin{equation}\label{eq.mainlimit}
        	 	\lim_{\delta\to 0}\lim_{k\to\infty}\int_{M_k\setminus \left(\bigcup_{y_i \in \mathcal{Y}\cap \partial N} B_{\delta}(y_i)\right)} |A_k|^n d \h^n = m\  \mathcal{A}(M)+ \sum_{j=1}^{J_N} \mathcal{A}(\Sigma_j)
        	 	\end{equation}
        	 	where $J_N$ denotes the (non-trivial) full bubbles forming at $y_i\in N\setminus\partial N$. 
        	 	It remains to understand what is happening on the small balls $B_\delta (y_i)$ as $\delta \to 0$ and for $y_i\in \partial N$. In order to study the behaviour here, we extract various point-scale sequences and look at their blow-up limits using Theorem \ref{thm:blowup}. During this process, we may iteratively pass to subsequences, but for the sake of simplicity we will not always state this explicitly. The important point here is that there will only be finitely many steps where this happens, so no diagonal argument is needed.  \\

        	 	For fixed $y\in \mathcal{Y}\cap \partial N$ consider the intersection $M_k\cap B_\delta(y)$: it is possible that this consists of more than one component but for $k$ sufficiently large there are at most $m$. Note that by the choice of $y$ we must have that $M_k \cap B_{r}(y)$ is unstable for all fixed $r>0$ and $k$ large enough. The rough plan from here is to extract point-scale sequences $(p_k, r_k)$ with $p_k\to y$ and $r_k\to 0$, and so that $\lambda_1(M_k \cap (B_{2r_k}(p_k)))\to -\infty$ for any such sequence. A bubbling argument as in \cite{BS17} will tell us that we can capture all the coalescing index in this way and that the process stops after at most $p-1$ such point-scale sequences were constructed.
        	 	This is the moral of Theorem \ref{thm:bubbles}, which we are indeed about to prove.

        	 	\begin{proof}[Proof of Theorem \ref{thm:bubbles}]
        	 		We prove the result for some fixed $y\in\mathcal{Y}\cap \partial N$ by constructing point-scale sequences as follows. Clearly we can repeat the steps precisely as below for each such $y$ so we do not concern ourselves with this.\\
        	 		
        	 		\textbf{The first point-scale sequence:} Let
        	 		\begin{equation*}
        	 		r^1_k = \inf \big\{r>0 \ : \ M_k\cap B_r(p) \text{ is unstable for some } p\in B_\delta(y)\cap M_k \big\}.
        	 		\end{equation*}
        	 		Note that $r^1_k$ defined above is strictly positive, and we can pick $p^1_k\in B_\delta(y)\cap M_k$ so that $M_k\cap B_{r^1_k+(r^1_k)^2}(p^1_k)$ is unstable. Notice that $M_k\cap B_{r^1_k/2}(z)$ is stable for all $z\in M_k\cap B_\delta (y)$ by definition.        		
        	 		We must have $r^1_k\to 0$ and $p^1_k\to y$ by the characterization of $\mathcal{Y}$ given in Lemma \ref{lemma:local}.  
        	 		
        	 		Based on these facts we have that:
        	 		\begin{enumerate}
        	 		\item[a)]{for every $k$ there exists at least one connected component of the intersection $M_k\cap B_{\delta/2}(p^1_k)$ satisfying the hypotheses of Theorem \ref{thm:blowup} with $\varrho_k = \delta/2$ and $r_k=r^1_k$. Thus we can perform a blow-up using normal coordinates centered at $p^1_k$, namely 
        	 			\begin{equation}\label{eq:blowup1}
        	 			\widetilde{M}^1_k \cap B_{\delta/r^1_k}(0) : = \frac{1}{r^1_k} \big(M_k\cap B_\delta(p^1_k)- p^1_k\big) 
        	 			\end{equation}
        	 			and such a sequence of connected components 
        	 			smoothly converges on every compact set with multiplicity one to a limit $\Sigma^1_1$ which is a non-trivial bubble or half-bubble and intersects the ball of radius two about the origin. In Case I of Theorem \ref{thm:blowup} the convergence happens locally smoothly on $\R^{n+1}$ and $\Sigma^1_1$ is always a non-trivial full bubble, while in Case II of Theorem \ref{thm:blowup} the convergence happens locally smoothly on some half-space and $\Sigma^1_1$ is either a bubble or half-bubble, non-trivial in either case.}
        	 		\item[b)]{for every $k$ all (other) connected components are still stable on all balls centered at any point $z\in M_k\cap B_{\delta/2}(p^1_k)$ and of radius $r^1_k/2$ so we can directly invoke Lemma \ref{lemma:local} which ensures smooth convergence with multiplicity one\footnote{When $n\geq 3$ there may be hyperplanes appearing as smooth limits here, but there is always at least one non-trivial (half-)bubble.} when we rescale according to equation \eqref{eq:blowup1}.}	
        	 		\end{enumerate}	
        	 		
        	  Let us denote by $\left\{\Sigma^1_i\right\}_{i\in I(1)}$ the finite collection of limit hypersurfaces in $\R^{n+1}$ we construct in this fashion. In particular, we obtain
        	 		\begin{equation}\label{eq.firstbubble}
        	 		\begin{aligned}
        	 		\lim_{R\to \infty} \lim_{k\to \infty} \int_{M_k\cap B_{Rr^1_k}(p^1_k)} |A_k|^n d \h^n &= \lim_{R\to \infty} \lim_{k\to \infty} \int_{\widetilde{M}^1_k\cap B^{\R^{n+1}}_{R}(0)} |\widetilde{A}^1_k|^n d \h^n\\
        	 		&= \lim_{R\to \infty}\sum_i \int_{\Sigma^1_i \cap B^{\R^{n+1}}_R(0)} |A|^n d \h^n \\
        	 		&= \sum_i\mathcal{A}(\Sigma^1_i).
        	 		\end{aligned}
        	 		\end{equation}
        	 		Theorem \ref{thm:blowup} also implies that $\lambda_1(M_k \cap (B_{2r_k}(p_k)))\to -\infty$. Using the half-space theorem \cite{HM90}, we further obtain that there is only one limit $\Sigma^1$ in dimension $n=2$, and that it must be a half-bubble in Case II. \\
        	 		
        	 		\textbf{The second point-scale sequence:} 
        	 		\begin{align*}
        	 		\mathfrak{C}^2_k:=\Big\{&B_r(p) \ : \ M_k\cap (B_r(p)\setminus B_{2r^1_k}(p^1_k)) \,\,\text{is unstable and $p\in B_\delta (y)\cap M_k$. Furthermore} \\
        	 		&\text{if $B_{2r}(p)\cap B_{2r^1_k}(p^1_k)\neq\emptyset$ then the connected component $Q$ of $B_{10r}(p)\cap M_k$}\\ 
        	 		&\text{containing $p$ is disjoint from $B_{2r^1_k}(p^1_k)$ and $Q\cap B_{r}(p)$ is itself unstable.}\Big\}
        	 		\end{align*}
        	 		
        	 		If $\mathfrak{C}^2_k = \emptyset$, regardless of the existence of further possible regions of coalescing index, we stop and leave it to the reader to check that in this case no further non-trivial bubbles can be found at $y$ (we advise the reader to check this upon a second reading). We may pass directly to the neck analysis at this point. 
        	 		
        	 		On the other hand, if $\mathfrak{C}^2_k \neq \emptyset$, we can set
        	 		\begin{eqnarray*}
        	 			r^2_k = \inf \Big\{r>0 \ : \  B_r(p)\in \mathfrak{C}^2_k\Big\}\end{eqnarray*}
        	 		and then, straight from the definition, we obtain $r^2_k\geq r^1_k$ 
        	 		 and we can find points $p^2_k\in B_\delta(y)\cap M_k$ such that $B_{r^2_k+(r^2_k)^2}(p^2_k)\in \mathfrak{C}^2_k$ and $M_k\cap (B_{r^2_k+(r^2_k)^2}(p^2_k)\setminus B_{2r^1_k}(p^1_k))$ is unstable. If $r^2_k$ does not converge to zero as one lets $k\to\infty$, then we stop the construction of point-scale sequences at $y$.
        	 		If instead $r^2_k\to 0$, then we ask whether or not 
        	 		\begin{equation}\label{eq.psinf}
        	 		\limsup_{k\to \infty}  \bigg(\frac{r^2_k}{r^1_k} + \frac{dist_g(p_k^1,p_k^2)}{r^2_k} \bigg)= \infty.
        	 		\end{equation}
        	 		If the answer to this question is negative (that is, if such a quantity stays bounded as $k\to\infty$) then there exists $C<\infty$ such that $B_{r^2_k}(p^2_k)\subset B_{Cr^1_k}(p^1_k)$ and $r^2_k\leq C r^1_k$ for all $k$. Thus we ignore this point-scale sequence since its blow-up limit is the same as for the previous point-scale sequence. We do however still keep track of the regions $B_{2r^2_k}(p^2_k)$ in order to find the next point-scale sequence and observe that, by appealing to the last part of Theorem \ref{thm:blowup} 
        	 		$$\lambda_1(M_k\cap (B_{2r^2_k}(p^2_k)\setminus B_{2r^1_k}(p^1_k))) \to -\infty.$$

        	 		If on the other hand \eqref{eq.psinf} holds, after passing to a subsequence for which the $\limsup$ is actually a limit, we distinguish two cases:\\
        	 		
        	 		\textbf{Case 1: The second fraction in \eqref{eq.psinf} tends to infinity.} In this case, the various (non-trivial) bubbles and half-bubbles are forming separately. We define $\varrho_k$ via 
        	 		$$2\varrho_k:=  dist_g(p^1_k,p^2_k)$$ and note that $\varrho_k/r^2_k\to \infty$. Therefore, for every $k$ the hypotheses of Theorem \ref{thm:blowup} apply to at least one component of $M_k\cap B_{\varrho_k}(p^2_k)$ and thus (arguing as we did when dealing with the first point-scale sequence) we get that
        	 		\begin{equation*}
        	 		\widetilde{M}_k^2 \cap B^{\R^{n+1}}_{\varrho_k/r^2_k}(0) := \frac{1}{r^2_k}\big(M_k\cap B_{\varrho_k}(p^2_k)- p^2_k\big)
        	 		\end{equation*}
        	 		smoothly converges (on compact sets) with multiplicity one to a collection $\left\{\Sigma^2_i\right\}_{i\in I(2)}$ which consist of either bubbles or a half-bubbles.\footnote{Once again, when $n\geq 3$ there may be hyperplanes appearing as smooth limits here, but there is always at least one non-trivial (half-)bubble.} Note that we also have 
        	 		$$\lambda_1(M_k\cap B_{2r^2_k}(p^2_k))=\lambda_1(M_k\cap (B_{2r^2_k}(p^2_k)\setminus B_{2r^1_k}(p^1_k)))\to \infty$$ by the same theorem. Again, if $n=2$, then the limit surface must be connected and a non-trivial half-bubble if we are in Case II of Theorem \ref{thm:blowup}. Moreover, as in \eqref{eq.firstbubble}, we have
        	 		\begin{equation}\label{eq.secbubble1}
        	 		\lim_{R\to \infty} \lim_{k\to \infty} \int_{M_k\cap B_{Rr^2_k}(p^2_k)} |A_k|^n \, d\h^n =\sum_i \mathcal{A}(\Sigma^2_i).
        	 		\end{equation}

        	 		\textbf{Case 2: The second fraction in \eqref{eq.psinf} is bounded by some $C_0<\infty$.} In this case, the second (non-trivial) bubble or half-bubble is forming on or near the first. We say that these bubbles are forming in a \emph{string}. Before describing the global and final picture, which shall be obtained (as always for a string of bubbles) by blowing-up centered at the first point $p^1_k$ in the string, let us analyze the structure of the second bubble in the string. 
        	 		
        	 		\

        	 		We consider the blow-up sequence with centers $p^2_k$ and scales $r^2_k$, that is we define
        	 		\begin{equation*}
        	 		\widetilde{M}_k^2 \cap B^{\R^{n+1}}_{\delta/r^2_k}(0) := \frac{1}{r^2_k}\big(M_k\cap B_\delta(p^2_k)- p^2_k\big).
        	 		\end{equation*}
        	 		
        	 		If we let $P_k$ denote the connected component of $M_k\cap B_{2r^2_k}(p^2_k)$ containing $p^2_k$, our definition of the class $\mathfrak{C}^2_k$ ensures that we have stability of $P_k$ on all balls centered at $z\in P_k$ and radius $r^2_k/2$, while we have instability (again, of $P_k$) on the ball of center $p^2_k$ and radius $r^2_k+(r^2_k)^2$. Hence, an argument along the lines of the proof of Theorem \ref{thm:blowup} ensures convergence of $P_k$ to a non-trivial (half-)bubble in $\R^{n+1}$, the convergence happening smoothly with multiplicity one. Also, notice (for completeness) that
        	 		at this scale some connected components of $\widetilde{M}_k^2$ (actually the components which correspond to the elements $\left\{\Sigma^1_i\right\}_{i\in I(1)}$) must converge to a hyperplane, but there can be no point(s) of bad convergence as far as the convergence of $P_k$ is concerned.

        	 		That being said, the usual scaling argument gives $$\lambda_1(M_k\cap(B_{2r^2_k}(p^2_k)\setminus B_{2r^1_k}(p^1_k)))\leq \lambda_1(Q^2_k\cap(B_{2r^2_k}(p^2_k)\setminus B_{Rr^1_k}(p^1_k)))\to -\infty$$ for any $R$.

        	 		A posteriori, it is clear (by virtue of the uniform boundedness of the ratio $dist_g(p_k^1,p_k^2)/r^2_k$) that we could equivalently blow-up around $p^1_k$, again with the same scale $r^2_k$ though, which would result in obtaining the very same limit hypersurfaces in $\R^{n+1}$ modulo Euclidean isometries. We will always employ this convention when dealing with strings of (half-)bubbles.
        	 			Denote by $\left\{\Sigma^2_i \right\}_{i\in I(2)}$ the set of minimal hypersurfaces one obtains by performing this blow-up, namely when letting $k\to\infty$ in
        	 			\[
        	 			\frac{1}{r^2_k}\big(M_k\cap B_\delta(p^1_k)- p^1_k\big).
        	 			\]
        	 			We further have
        	 		\begin{equation}\label{eq.secbubble2}
        	 		\lim_{\delta_1\to 0}\lim_{R\to \infty}\lim_{k\to \infty} \int_{M_k\cap B_{Rr^2_k}(p^1_k)\setminus B_{\delta_1 r^2_k}(p^1_k)} |A_k|^n \,d\h^n =  \sum_j \mathcal{A}(\Sigma^2_j),
        	 		\end{equation}
        	 		while note that we have not yet controlled the term
        	 		\begin{equation}\label{eq.neck1}
        	 		\lim_{\delta_1\to 0}\lim_{R\to \infty}\lim_{k\to \infty}\int_{M_k\cap B_{\delta_1 r^2_k}(p^1_k)\setminus B_{Rr^1_k}(p^1_k)} |A_k|^n d \h^n.
        	 		\end{equation}
        	 		However, we will see that $M_k$ is a neck or half-neck of order $(\eta (R,\delta_1), L(R,\delta_1))$ in this region. We will deal with neck regions like this in the following section, where we will see that the above limit actually vanishes.
        	 		
        	 		When $n=2$, the half-space theorem shows that a (half-)plane and a non-trivial component would necessarily have to intersect, a contradiction. Hence, Case 2 cannot occur in dimension $n=2$; we necessarily have to be in Case 1.\\
        	 		
        	 		\textbf{Further point-scale sequences:}
        	 		We continue with the above described scheme iteratively. Suppose we have extracted $\ell-1$ point-scale sequences (including the ones that we have ignored) for $y$. Then we  continue (or not) under the following rules: let $U^{\ell-1}_k = \bigcup_{s=1}^{\ell-1} B_{2r^s_k }(p_k^s)$ and define an admissible class of balls with $r^{\ell -1}_k<\delta$ and
        	 		\begin{align*}
        	 		\mathfrak{C}^\ell_k:=\Big\{&B_r(p) \ : \ M_k\cap (B_r(p)\setminus U^{\ell -1}_k  \,\,\text{is unstable and $p\in B_\delta (y)\cap M_k$. Furthermore } \\
        		&\text{if $B_{2r}(p)\cap U^{\ell-1}_k\neq\emptyset$ then the connected component $Q$ of $B_{10r}(p)\cap M_k$} \\
        		&\text{containing $p$ is disjoint from $U^{\ell-1}_k$ and $Q\cap B_{r}(p)$ is itself unstable.} \Big\}.
        	 		\end{align*}
        	 		Now if $\mathfrak{C}^\ell_k =\emptyset$ we stop the process here. If not, set
        	 		\begin{equation*}
        	 		r^\ell_k = \inf \Big\{r>0 \ : \ B_r(p)\in \mathfrak{C}^{\ell}_k\Big\}.
        	 		\end{equation*} and pick $p^{\ell}_k\in B_\delta(y)\cap M_k$ so that $B_{r^\ell_k +(r^{\ell}_k)^2}(p^\ell_k)\in \mathfrak{C}^\ell_k$ (in particular $M_k \cap\big( B_{r^\ell_k +(r^{\ell}_k)^2}(p^\ell_k)\setminus U^\ell_k\big)$ is unstable). 
        	 		
        	 		Again, we must have $r^\ell_k\geq r^{\ell-1}_k$ (since the class of admissible balls gets smaller). If $r^\ell_k \not\to 0$ then we discard $(p_k^\ell, r_k^\ell)$ and the process stops. If we do have $r_k^\ell \to 0$, then we ask whether or not it is true that
        	 		\begin{equation}\label{eq.psinf2}
        	 		\min_{i=1,\dots, \ell-1}\limsup_{k\to \infty}  \bigg(\frac{r^\ell_k}{r^{i}_k} + \frac{dist_g(p_k^\ell,p_k^i)}{r^\ell_k} \bigg) = \infty.
        	 		\end{equation}
        	 		If the answer is negative, then there exists $C<\infty$ such that $B_{r^\ell_k}(p^\ell_k)\subset B_{Cr^s_k}(p^s_k)$ \emph{and} $r^s_k\leq C r^s_k$ for all $k$ and some $s<\ell$ so that any blow-up corresponding to the sequence will yield a limit scenario that has already been captured at an earlier step. As before we keep track of the regions $B_{2r^\ell_k}(p^\ell_k)$ and we also note that 
        	 		$$\lambda_1(M_k \cap (B_{2r^\ell_k}(p^\ell_k)\setminus U^{\ell-1}_k))\to -\infty$$ by the last part of Theorem \ref{thm:blowup}.
        	 		
        	 		If on the other hand \eqref{eq.psinf2} holds, then after passing to a subsequence for which the $\limsup$ in \eqref{eq.psinf2} is actually a limit, we distinguish the following cases:\\
        	 		
        	 		\textbf{Case 1': The second fraction in \eqref{eq.psinf2} tends to infinity for all $i$.} In this case, the (half-)bubble is forming separately from the previously extracted (half-)bubbles and strings of bubbles. We can follow Case 1 from above, except that this time we blow-up centered at $p^\ell_k$ in a ball of radius $\varrho^\ell_k$ satisfying 
        	 		$$2\varrho^\ell_k = \min_{i<\ell} dist_g(p^\ell_k, p^i_k).$$ We blow-up at scale $r^\ell_k$ to obtain some collection $\left\{\Sigma^\ell_i\right\}_{i\in I(\ell)}$ whose elements are either bubbles or half-bubbles, non-trivial in both cases, with
        	 		$$\lambda_1(M_k \cap (B_{2r^\ell_k}(p^\ell_k)\setminus U^{\ell-1}_k))=\lambda_1(M_k\cap B_{2r^\ell_k}(p^\ell_k))\to -\infty$$ and
        	 		\begin{equation}\label{eq.lthbubble1}
        	 		\lim_{R\to \infty} \lim_{k\to \infty} \int_{M_k\cap B_{Rr^\ell_k}(p^\ell_k)} |A_k|^n \,d\h^n = \sum_{i\in I(\ell)}\mathcal{A}(\Sigma^\ell_i).
        	 		\end{equation}
        	 		
        	 		\textbf{Case 2': The second fraction in \eqref{eq.psinf2} is bounded for some $i$.} We first note that if this is the case for some index $i$ and $\Sigma^i$ is an element in a string of (half-)bubbles, then the property also holds for any other index $s$ corresponding to other non-trivial (half-)bubbles $\Sigma^s$ in the same string.  See \cite{BS17}, p.~4389, for more details on this observation. It is possible that the second fraction in \eqref{eq.psinf2} could be bounded for indices corresponding to elements in two or more distinct strings, i.e. ~at \emph{this scale} these previously distinct strings appear together, in which case we refer to the union of these strings as one new string. Note that this shows in particular that we must have $r^\ell_k/r^i_k\to \infty$ for any $i$ corresponding to elements in this string. Call the \emph{earlier} indices of the string $i_1,\ldots,i_m$ (so that $\ell=i_{m+1}$).
        	 		
        	 		Similarly as in Case 2 above, we now blow-up centered at the first point in the string (which is $p^{i_1}_k$), but at a scale $r^\ell_k$. As above, this is impossible when $n=2$, while for $n\geq 3$ Theorem \ref{thm:blowup} yields a collection of non-trivial (half-)bubbles $\Sigma^\ell_j$ and also implies $$\lambda_1 (M_k\cap (B_{2r^\ell_k}(p^\ell_k)\setminus U^{\ell-1}_k))\leq\lambda_1(M_k\cap(B_{2r^\ell_k}(p^\ell_k)\setminus B_{Rr^{i_m}_k}(p^{i_1}_k)))\to -\infty$$ as well as
        	 		\begin{equation}\label{eq.lthbubble2}
        	 		\lim_{\delta_1\to 0 }\lim_{R\to \infty}\lim_{k\to \infty} \int_{M_k\cap B_{Rr^\ell_k}(p^{i_1}_k)\setminus B_{\delta_1 r^{i_m}_k}(p^{i_1}_k)} |A_k|^n d \h^n = \mathcal{A}(\Sigma_\ell).  
        	 		\end{equation} 
        	 		
        	 		Again, there is a neck region between the previously largest (half-)bubble in the string with index $i_m$ and the new $\ell$-th (half-)bubble, which we still have not controlled. We will deal with it in the following section.

        	 		Notice that at each stage of the point-scale selection and blow-up process we are accounting for a new subdomain on $M_k$ where $\lambda_1 \to -\infty$ thus this process stops after at most $p-1$ iterations, until we have exhausted all point-scale sequences. Each new scale yields either a non-trivial bubble forming on one of the previous bubbles (in a string), or it is occurring on its own scale. Each time we are accounting for all the total curvature \emph{except} on the (half-)neck regions between consecutive elements in a string as in \eqref{eq.neck1}. 
        	 		
        	 		If we take a distinct point-scale sequence $(q_k,\varrho_k)$ as in the final clause of the statement of the theorem, then if we blow-up at this scale and we end up with something non-trivial in the limit, then we must have captured some more coalescing index in an admissible ball, but this cannot happen since by construction we have exhausted \emph{all unstable regions} in the process.
        	 	\end{proof}
        	 	
        	 	Another important point of notation is determined in the next definition. Notice that at each bubble point $y$ we can classify the point-scale sequences into finitely many different strings, recalling that when $n=2$ each string has only one element, or equivalently there are no strings. 
        	 	
        	 	\begin{definition}\label{def:int}
        	 		Given a point-scale sequence $(p^i_k, r^i_k)$ corresponding to a non-trivial bubble forming at $y$, and in a string whose first non-trivial bubble forms at the points $p^1_k$ we are in one of three scenarios
        	 		\begin{enumerate}
        	 			\item This is the final (or only) bubble in a string and there are no other strings forming at $y$. Set $\xi_k = 1$ (so in fact it is independent of $k$). 
        	 			\item This is the final (or only) bubble in a string and the closest distinct string forming at $y$ has its first bubble forming at points $q^j_k$ and final bubble scale $s^j_k$. Setting $\xi_k = dist_g(p^i_k,q^j_k)$ we have $\xi_k/(r^i_k+s^j_k) \to \infty.$
        	 			\item This is not the final bubble in a string, and the next non-trivial bubble occurs at scale $\xi_k = r^{i+1}_k$ and we have $\xi_k/r^i_k \to \infty$.  
        	 		\end{enumerate}
        	 		In any of the above cases we say that $\xi_k$ is the \emph{intermediate scale}.
        	 		
        	 		The \emph{neck region}\footnote{A point of notation; the neck region in a bubbling analysis should not be confused with the (perhaps more geometric) neck appearing as part of the bubble. For instance if one considers a blown-down catenoid in Euclidean space, centered at the origin and converging smoothly to a double plane away from the origin (say at scale $r_k \to 0$), the neck region would refer to $B_{\delta}\setminus B_{Rr_k}$ for $k$ large, $R$ large and $\delta$ small. This should not be confused with the degenerating `neck' of the catenoid.}   of this bubble scale is defined to be 
        	 		$$M_k\cap (B_{\delta \xi_k}(p^1_k)\setminus B_{R r^i_k}(p^1_k))$$
        	 		for $\delta$ sufficiently small, $R$, $k$ sufficiently large. 
        	 	\end{definition}
        	 	
        	 	The neck regions are precisely those that we have not analyzed yet; we will deal with these in the next section.  
        A first corollary of Theorem \ref{thm:bubbles}, which we will further improve, is the following. 
        	 	
        	 	\begin{cor}\label{cor:bubbles}
        	 		With the setup as in Theorem \ref{thm:bubbles}, denoting by $\{\Sigma_j\}_{j=1}^J$ the collection of all the non-trivial bubbles and half-bubbles, we have
        	 		\begin{equation*}
        	 		\lim_{k\to\infty}\mathcal{A}(M_k) \geq m\mathcal{A}(M) + \sum_{j=1}^J \mathcal{A}(\Sigma_j).\\
        	 		\end{equation*}
        	 	\end{cor}
        	 	
        	 	\begin{proof}
        	 		This follows directly from combining \eqref{eq.mainlimit}, \eqref{eq.firstbubble}, \eqref{eq.secbubble1}, \eqref{eq.secbubble2}, \eqref{eq.lthbubble1}, and \eqref{eq.lthbubble2} and noting that the regions considered in these equations are mutually disjoint. 
        	 	\end{proof}

        	 	\section{Neck analysis}\label{sec:neck}
        	 	
        	 	The goal of this section is to finish the proof of Theorem \ref{thm:quant} improving the inequality in Corollary \ref{cor:bubbles} to an equality, by showing that no further total curvature can concentrate in the neck regions (see Definition \ref{def:int}).

        	 	Thus we content ourselves with proving that the limits of the form \eqref{eq.neck1} are zero. In fact we will prove a little more than this: that the ends of each bubble or half-bubble must be parallel to $T_y M$ (in a suitable sense), see Lemma \ref{lem:thm1}. We will prove that in either case, for $\delta$ small enough and $R,k$ large enough, these regions are described in precisely three different scenarios which we describe now.

        	 	\begin{definition}\label{neckdef}
        	 		Let $M\in \mathfrak{M}$. For $p\in M$ we say that $M\cap B_\delta(p)\setminus B_{\eps}(p)$ is a \emph{neck of order} ($\eta, L$) if we have $\eps < \delta/4$ and $M\cap (B_\delta\setminus B_{\eps})$ is uniformly graphical over some plane which we may assume (after a rotation) to be defined by $\{x^{n+1}=0\}$ in normal coordinates about $p$. More precisely there exist functions $u^1,\dots u^L$ such that 
        	 		\begin{equation*}
        	 		M\cap (B_\delta \setminus B_{\eps}) = \bigcup_{i=1}^L \; \{(x^1,\dots , x^n, u^i(x^1,\dots, x^n))\}
        	 		\end{equation*}
        	 		and also 
        	 		\begin{equation}
        	 		\eta:=\sup_{i=1,\ldots,L} \sup_{x\in B_\delta \setminus B_{\eps}} |\D u^i (x)| + \sup_{\eps< \varrho < \frac{\delta}{2}} \int_{M \cap (B_{2\varrho}\setminus B_\varrho)} |A|^n d\h^n < \infty. 
        	 		\end{equation}
        	 	\end{definition}
        	 	
        	 	For the below definitions we will deal with Euclidean balls in Fermi-coordinate neighborhoods (see the beginning of Section \ref{sec:bubbling}). For simplicity we now translate these neighborhoods so that $\Pi_1 = \{x^1 \geq 0\}$ from now on. In particular we will deal with Euclidean half-balls
        	 	$$B_\delta^+=B^{\R^{n+1}}_\delta (0)\cap \Pi_1\subset  [0,\gamma)\times B^{\R^n}_{\gamma}(0)=C^{\R^{n+1}}_\gamma(0),$$ and notice that these correspond to some simply connected domain (though not a geodesic ball) in $N$ which we denote $\widehat{B}^N_\delta(p)\subset C_{\gamma}(p)$ for all $p\in \partial N$.
        	 	
        	 	\begin{definition}\label{neckdef_nonc_boundary}
        	 		Let $M\in \mathfrak{M}$. For $p\in \partial N\cap M$ we say that $M\cap B^{+}_\delta\setminus B^{+}_{\eps}$ is a \emph{half-neck of order} ($\eta, L$) if we have $\eps < \delta/4$ and $M\cap (B_\delta^+\setminus B_{\eps}^+)$ is uniformly graphical over some plane which we may assume (after a rotation) to be defined by $\{x^{n+1}=0\}$ in Fermi-coordinates about $p$. More precisely there exist functions $u^1,\dots u^L$ such that 
        	 		\begin{equation*}
        	 		M\cap (B_\delta^+ \setminus B_{\eps}^+) = \bigcup_{i=1}^L \; \{(x^1,\dots , x^n, u^i(x^1,\dots, x^n))\}
        	 		\end{equation*}
        	 		and also 
        	 		\begin{equation}
        	 		\eta:=\sup_{i=1,\ldots,L}\sup_{x\in B^+_\delta \setminus B^+_{\eps}} |\D u^i|+ \sup_{\eps< \varrho < \frac{\delta}{2}} \int_{M \cap (B^+_{2\varrho}\setminus B^+_\varrho)} |A|^n \, d\h^n  < \infty. 
        	 		\end{equation}
        	 	\end{definition}
        	 	
        	 	\begin{rmk}\label{rmk_reflection} 
        	 		
        	 		We can reflect a half-neck of order $(\eta, L)$ across $\Pi=\{x^1 =0\}$ to obtain a neck of order $(2\eta, L)$ at $p$ as per Definition \ref{neckdef}, except now the minimal surface lies inside a Riemannian manifold with a Lipschitz-regular metric across $\Pi$ -- in particular the full neck we obtain will be at least $W^{2,\infty}$-regular, and in general no more (though smooth away from $\Pi$). We therefore have more than enough regularity to analyze the local properties of a half-neck via the reflected full neck.  
        	 		
        	 		Consider the reflection $\sigma$ about the plane $\{x^1=0\}$ in $\R^{n+1}$ and define a metric on $\check{C}^{\R^{n+1}}_{\gamma}=B_\gamma^{\R^n}(0)\times (-\gamma,\gamma)$ via 
        	 		$$\check{g} = \twopartdef{g}{x^1\geq 0}{\sigma^{\ast}g}{x^1<0.}$$
        	 		This metric is smooth away from $\{x^1=0\}$ and Lipschitz on $\check{C}^{\R^{n+1}}_{\gamma}$ (it is smooth if $\partial N$ is totally geodesic in $N$ near $p$).
        	 		
        	 		Now, if $u$ is a free boundary minimal graph over $\Omega \subset C^{\R^{n+1}}_{\gamma}\cap \{x^{n+1}=0\}$, describing some piece of a free boundary minimal surface $\Sigma \subset C^{\R^{n+1}}_{\gamma}$ then we have that $\Sigma$ is parametrized by 
        	 		$$(x^1,\dots,x^n)\mapsto (x^1,\dots, x^n, u(x^1,\dots , x^n))$$
        	 		and 
        	 		$$\frac{\partial u}{\partial x^1}(0,x^2,\dots x^n) = 0$$
        	 		for all boundary points where $u$ is defined i.e. on $\Omega\cap \{x^1 = 0\}$. 
        	 		
        	 		Defining, for $x=(x^1,\dots x^n)\in \Omega$,   
        	 		$$\check{u}(x) = \twopartdef{u(x)}{x_1\geq 0}{u\circ \sigma (x)}{x_1<0}$$
        	 		we see that $\check{u}$ is a $C^1$ minimal graph over $\check{\Omega}=\Omega\cup \sigma(\Omega)$ describing $\check{\Sigma} = \Sigma\cup \sigma(\Sigma)$ with respect to a Lipschitz ambient metric. Thus in particular $\check{u}\in W^{2,\infty} (\check{\Omega})$ since $u$ was smooth up to the boundary and is $C^1$ across the boundary. We cannot improve on the regularity of $\check{u}$ since it solves an elliptic equation of the form $L\check{u}=f(\check{u},\D \check{u})$ where $L$ is an elliptic operator whose coefficients are $C^0$-close to those of the Euclidean Laplacian and $f\in L^\infty$ but no more, unless $\partial N$ is totally geodesic near $p$ in which case the coefficients of $L$ and $f$ become smooth, and we conclude full regularity via a boot-strapping argument.
        	 		
        	 		%
        	 	\end{rmk}
        	 	
        	 	The below definition looks similar to the preceding one, however it is in fact encapsulating something very different. This will correspond to a bubble with a compact boundary, and in fact its ends will be graphical over $T_y \partial N$ whereas the previous definition corresponds to a bubble with a non-compact boundary component whose ends will be graphical over a plane orthogonal to $\Pi$. 
        	 	
        	 	\begin{definition}\label{neckdef_c_boundary}
        	 		Let $M\in \mathfrak{M}$. For $p\in \partial N\cap M$ we say that $M\cap B^{+}_\delta\setminus B^{+}_{\eps}$ is a \emph{compact neck of order} ($\eta, L$) if we have $\eps < \delta/4$ and $M\cap (B_\delta^+\setminus B_{\eps}^+)$ is uniformly graphical over $\{x^{1}=0\}$ in Fermi-coordinates about $p$ (see the beginning of section \ref{sec:bubbling}). More precisely there exist functions $u^1,\dots u^L$ such that 
        	 		\begin{equation*}
        	 		M\cap (B_\delta^+ \setminus B_{\eps}^+) = \bigcup_{i=1}^L \; \{(u^i(x^2,\dots,x^{n+1}),x^2,\dots ,x^{n+1})\}
        	 		\end{equation*}
        	 		and also 
        	 		\begin{equation}
        	 		\eta:=\sup_{i=1,\ldots,L} \sup_{x\in B^+_\delta \setminus B^+_{\eps}}  |\D u^i | + \sup_{\eps< \varrho < \frac{\delta}{2}} \int_{M \cap (B^+_{2\varrho}\setminus B^+_\varrho)} |A|^n \, d\h^n< \infty. 
        	 		\end{equation}
        	 	\end{definition}
        	 	
        	 	We can now continue the proof of Theorem \ref{thm:quant} by studying regions of the form $M_k\cap (B_{\delta}(p^y_k)\setminus B_{R r^y_k}(p^y_k))$ -- each neck region is of this form, and many different necks will appear in general, but the usual covering argument will allow us to study only one such region in detail. This result is similar in spirit to \cite[Theorem 1.1]{W15}.

    	 	\begin{lem}\label{lem:thm1} 
        	 		With the setup as Theorems \ref{thm:quant}, and \ref{thm:bubbles}, given $y\in \mathcal{Y}\cap\partial N$ then each bubble that appears has ends which are parallel to $T_y M$ in the following sense:
        	 		\begin{enumerate}
        	 		\item{Given a point-scale sequence $(p^i_k, r^i_k)$ corresponding to a non-trivial bubble, we let $\xi_k$ be the intermediate scale (see Definition \ref{def:int}).        		
        	 		Then for $\delta$ sufficiently small, $R$, $k$ sufficiently large, the neck region $$M_k\cap (B_{\delta \xi_k}(p_k)\setminus B_{R r_k}(p_k))$$ is described by a neck, half-neck, or compact neck of order $(\eta, L)$ with 
        	 		$$\lim_{\delta \to 0}\lim_{R\to \infty} \lim_{k\to \infty} \eta \to 0.$$}
        	 		\item{If we let $s_k\to 0$ be any sequence then by blowing up at coordinates centered at $y$ to give $\widehat{M}_k \subset B_{\delta/s_k}(y)$ via 
        	 		$$\widehat{M}_k = \frac{1}{s_k}(M_k\cap B_{\delta}(y)),$$
        	 		the limit is a collection of bubbles (possibly all hyperplanes or even empty) all of whose ends are parallel to $T_y M$. }
        	 	\end{enumerate}
        	 		\end{lem}
        	 	
        	 	The above Lemma nearly completes the proof of Theorem \ref{thm:quant} and we also obtain the following possible types of behaviour.\\
        	 	
        	 	\framebox[1.8cm]{Case A:} If $y\in N\setminus \partial N$ then we are in the setting of \cite{BS17} and see only full bubbles whose ends are parallel to $T_y M$. Obviously, this case corresponds to Definition \ref{neckdef}. \\
        	 	
        	 	\framebox[1.8cm]{Case B:} If $y\in \partial N\cap \partial M$ then the above Lemma tells us that
        	 	the bubbles and half-bubbles that occur have ends parallel to $T_y M$, and as a result the half-bubbles have ends which are orthogonal to $\Pi$. The possible scenarios correspond to Definition \ref{neckdef} and Definition \ref{neckdef_nonc_boundary}.
        	 	
        	 	In particular, notice that the free boundary of each half-bubble has at least two non-compact components. This last assertion follows from Corollary \ref{cor:oneend} (which can be rephrased as: if a half-bubble has exactly one non-compact boundary component then it must be a hyperplane).  \\

        	 	\framebox[1.8cm]{Case C:} If $y\in \partial N\setminus \partial M$ then this time the above Lemma tells us that the only bubbles and half-bubbles that occur have ends which are parallel to $\Pi$. The possible scenarios correspond to Definition \ref{neckdef} and Definition \ref{neckdef_c_boundary}.
        	 	Notice that in this case the free boundary of each half-bubble is compact. This cannot occur under the assumption $(\textbf{P})$. In this case we also trivially have that, for all $\eta$ sufficiently small $\partial M_k\cap B_\eta(y) \to 0$ as varifolds.  \\
        	 	
        	 	In particular notice that in Cases B and C, it is possible for both full and half-bubbles to occur at a single point $y$.

       	\begin{proof}[Proof of Lemma \ref{lem:thm1}] 
        	 		In fact the second statement of the Lemma is relatively straightforward given the first, so we shall not discuss it. The more technical part is the first one -- which shows that the minimal hypersurfaces in the neck regions are behaving exactly like the ends of the bubbles -- and thus these ends are parallel to any larger-scale blow-up we wish to execute. 
        	 		
        	 		In Case I of Theorem \ref{thm:blowup}, which corresponds to an enclosed full bubble, the neck analysis as carried out in Section 4 of \cite{BS17} holds in precisely the same way, so we will not give any further details here and Lemma \ref{lem:thm1} is proved in this case. 
        	 		
        	 		It remains to deal purely with the setting of Case II in Theorem \ref{thm:blowup}, and in fact we only deal with the case when the blow-up limit is a half-bubble (since the full-bubble case has been dealt with implicitly above).

        	 		Thus, employing the very same notation as in the statement of Theorem \ref{thm:blowup}, we see that the ball of radius $(2c+1)r^i_k$ about $q_k$ contains the bubble region and we therefore take a Fermi-coordinate neighborhood about $q_k$ as described at the beginning of Section \ref{sec:bubbling}. Following the scheme in \cite{BS17} p.~4390 in the proof of Claim 1, will see that the neck region is either contained in a \emph{half-neck} or a \emph{compact neck} of order $(\eta, L)$ (after a rotation) and 
        	 		\[
        	 		\lim_{\delta\to 0}\lim_{R\to \infty}\lim_{k\to \infty} \eta = 0.
        	 		\]
        	 		
        	 		If the half-bubble has ends which are parallel to $\Pi$ then this argument is similar to Section 4 of \cite{BS17}. The blow-up argument in the neck region is more straightforward in this case since everything is happening above (and converging to) a fixed plane $\{x^1 = 0\}$, thus in particular no maximum-principle argument is required to prove that we have a compact neck of order $(\eta, L)$ with $\eta\to 0$. 
        	 		
        	 		If the half-bubble has ends which are orthogonal to $\Pi$ then using Section 3 of \cite{ACS17}, we construct a free boundary foliation near $q_k$ which will allow us to change coordinates and run a maximum principle argument as in \cite{BS17} p.~4391 to conclude that the neck region is a half-neck of order $(\eta ,L)$ and with $\eta \to 0$. We leave the details to the reader, noting that a Hopf-boundary maximum principle is required in this setting.       	
        	 	\end{proof}

        	 	\begin{proof}[Proof of Theorem \ref{thm:quant}] 
        	 		We will show that no total curvature is lost in the neck regions. The claimed curvature quantization result then follows from this fact and Corollary \ref{cor:bubbles}. 
        	 		
        	 		When we are dealing with a neck region or a compact neck region then the argument is exactly as in \cite{BS17}, p.~4392--4395. In the case that we have a half-neck the reflection procedure as described in Remark \ref{rmk_reflection} turns the \emph{half-neck} into a \emph{neck} as per Definition \ref{neckdef}, still with $\eta \to 0$. Notice that the graphs will no longer be smooth across $\{x^1 =0\}$, but they will be $W^{2,\infty}$-regular which allows us to run the argument as in \cite{BS17}, p.~4392--4395, to conclude that no total curvature is lost in such regions (in exactly the same fashion -- notice that this result does not require more regularity than we have). 
        	 		
        	 		The final statement of the theorem follows from a covering argument: since $\left\{M_k\right\}$ is converging (up to extracting subsequences, and possibly after a blow-up in the case of the bubble regions) to uniquely determined limit objects, given $k_1, k_2$ large integers we can cover the bubble regions, the neck regions, and large-scale regions of both $M_{k_1}$ and $M_{k_2}$ with finitely many charts that are all pairwise diffeomorphic. Thus, it follows that $M_{k_1}$ shall be diffeomorphic to $M_{k_2}$, which means that all elements of the sequence $\left\{M_k\right\}$ are eventually pairwise diffeomorphic to one-another. 
        	 	\end{proof}

		\section{Geometric applications}\label{sec:appl}
		
		Let us recall from Theorem \ref{thm:quant} that, under the assumption that a sequence of free boundary minimal hypersurfaces $\left\{M_k \right\}$ converge to $M$ with multiplicity $m\geq 1$ we have the quantization formula
		\[
		\lim_{k\to\infty}\mathcal{A}(M_k) =m\mathcal{A}(M)+ \sum_{j=1}^J\mathcal{A}(\Sigma_j).
		\]

		We now discuss how to turn it into a topological relation involving $M_k, M$ and the non-trivial bubbles (or half-bubbles) that arise in the blow-up procedure. A key point in that respect is the following assertion, of independent interest, which follows from \cite[Theorem 29]{ACS17}.
		
		\begin{prop}\label{prop:boundaryconv}
		In the setting of Theorem \ref{thm:quant}, let us regard $\partial M_k$ and $\partial M$ as integer $(n-1)$-dimensional varifolds $\mu_k$ and $\mu$, respectively, where it is understood that $\mu_k$ has unit multiplicity while $\mu$ has multiplicity $m$ as in the convergence statement. Then $\mu_k \to \mu$ in the standard weak sense of varifolds, namely for any $f\in C(\mathbb{G})$
		\[
		\lim_{k\to\infty}\int_{\mathbb{G}} f\,d\mu_k =\int_{\mathbb{G}} f\,d\mu
		\]
		where $\mathbb{G}$ denotes the Grassmann bundle of unoriented $(n-1)$-planes over the ambient boundary $\partial N$. 
		\end{prop}
		
		\begin{proof}
		The result is trivial if $\mathcal{Y}\cap\partial N=\emptyset$ (for in that case Theorem \ref{thm:quant} ensures smooth convergence, possibly multi-sheeted, of the boundaries), so let us assume on the contrary that this intersection is not empty. 
		We claim that for all $y\in \mathcal{Y}\cap \partial N$ 
			\begin{equation*}
			\lim_{\varrho \to 0} \lim_{k\to\infty}\int_{B_\varrho (y)\cap \partial N} d\mu_k = 0,
			\end{equation*}
			and in fact that there exists some constant $\sigma$, depending only on $(N,g)$, $M$  and $m$, such that for all $k$ sufficiently large 
			\begin{equation}\label{eq:o(2)}
			\int_{B_\varrho (y)\cap \partial N} d\mu_k \leq \sigma \varrho^{n-1}.
			\end{equation}
			Assuming the claim and given $f\in  C(\mathbb{G})$, let $\eta_\varrho \in C^{\infty}_c(N)$ be a smooth non-negative ambient function which equals one in all geodesic balls centered at a point of $\mathcal{Y}\cap \partial N$ and of radius $\varrho$, and is supported in the union of the geodesic balls with the same centers and radii $2\varrho$. Set $f_\varrho = (1-\eta_\varrho)f$, we can write by the triangle inequality 
			\[
			\left| \int_{\mathbb{G}} f\,d\mu_k -\int_{\mathbb{G}} f\,d\mu \right|\leq \left| \int_{\mathbb{G}} f\,d\mu_k -\int_{\mathbb{G}} f_\varrho\,d\mu_k \right|
			+\left| \int_{\mathbb{G}} f_\varrho\,d\mu_k -\int_{\mathbb{G}} f_\varrho\,d\mu \right|
			+\left| \int_{\mathbb{G}} f_\varrho\,d\mu -\int_{\mathbb{G}} f\,d\mu \right|.
			\]
			By \eqref{eq:o(2)} the first summand satisfies 
			\[\lim_{k\to \infty} \left| \int_{\mathbb{G}} f\,d\mu_k -\int_{\mathbb{G}} f_\varrho\,d\mu_k \right| \leq \sup |f|\cdot \sigma(2\varrho)^{n-1}. 
			\]
			By smooth graphical convergence away from $y\in\mathcal{Y}$, and for any $\varrho >0$, the second summand satisfies
			\[
			\lim_{k\to \infty}\left| \int_{\mathbb{G}} f_\varrho\,d\mu_k -\int_{\mathbb{G}} f_\varrho\,d\mu \right|\leq \lim_{k\to \infty} \sup |f|\cdot |[\mu_k-\mu](\partial N \setminus B_{\varrho}(y))| = 0.
			\]
			Since $\mu$ has finite mass (for $M$ is a smooth free boundary minimal surface) we have
			\[
			\lim_{\varrho\to 0}\left| \int_{\mathbb{G}} f_\varrho\,d\mu -\int_{\mathbb{G}} f\,d\mu \right|=0
			\]
				by Lebesgue's dominated convergence theorem.
		Combining these three simple facts, we obtain the varifold convergence of $\mu_k$ to $\mu$.
		
		\
			It remains to justify the claim \eqref{eq:o(2)}. Pick $y\in \mathcal{Y}\cap \partial N$ and for $\varrho$ sufficiently small choose a Fermi-coordinate neighborhood so that $B_{2\varrho} (y) \subset C_{3\varrho} (y)$ (where we are employing the notation presented at the beginning of Section \ref{sec:bubbling}). On this neighborhood, set $X_1 = -\partial_{x^1}$ and choose a non-negative $\phi\in C_c^\infty (B_{2\varrho} (y))$ so that $\phi \equiv 1$ on $B_\varrho (y)$ and $|\D \phi| \leq 4/\varrho$. Notice that on $C_{3\varrho}(y)$ there exists $\sigma_1>0$ (only depending on $(N,g)$) so that $|\D X_1 | \leq \sigma_1$. Set $X= \phi X_1$ and using the fact that $M_k$ is a properly embedded free boundary minimal hypersurface, so that the outward unit conormal $\nu_k$ of $M_k$ satisfies (in this neighborhood) $\nu_k = X_1$, we see that for all $k$ sufficiently large 
			\begin{align*}
			\int_{B_\varrho (y)\cap \partial N} d \mu_k &\leq \int_{B_{2\varrho}(y)\cap \partial M_k} \phi \,d\h^{n-1} = \int_{B_{2\varrho}(y)\cap \partial M_k} g(X, \nu_k)\, d\h^{n-1}\\ 
			&= \left| \int_{M_k\cap B_{2\varrho}(y)} div_{M_k}(X)\, d\h^n \right| \leq \left(2\sigma_1+\frac{8}{\varrho}\right)\int_{M_k\cap B_{2\varrho}(y)} d \h^n \leq \sigma \varrho^{n-1}
			\end{align*}
			where in the final step we have used that $M_k\to mM$ as varifolds as well as the monotonicity formulae for free boundary minimal hypersurfaces (see e.g. \cite[Corollary 16]{ACS17}), which holds even if $M$ fails to be properly embedded at $y$. This completes the proof.
	\end{proof}

\begin{cor}
	\label{cor:boundlength}
		Let $2\leq n\leq 6$ and $(N^{n+1},g)$ be a compact Riemannian manifold with boundary. Given $\Lambda, \mu\in \R_{\geq 0}$ and $p\in\N_{\geq 1}$ there exists a constant $C=C(p,\Lambda,\mu, N, g)$ such that $\h^{n-1}(\partial M)\leq C$ for any $M\in \mathfrak{M}_p(\Lambda,\mu)$.
		\end{cor}
	
		Now, let us specify this result to the case $n=2$: if $M\subset N$ is a free boundary minimal surface (with respect to $\partial N$) we can regard its geodesic curvature (namely: the geodesic curvature of $\partial M$ as a subset of $M$) as a well-defined function on the line bundle $\mathbb{G}$. 
		
		 \begin{rmk}\label{rem:frame}
		 	When $M\subset N$ is a free boundary minimal surface we shall denote by $\left\{\tau,\nu,\epsilon \right\}$ a local orthonormal frame at any boundary point, so that $\tau$ is tangent, $\nu$ is outward-pointing and conormal, and $\epsilon$ is normal to $M\subset N$. Furthermore, notice that $\left\{\tau,\epsilon\right\}$ can be extended to a local tangent frame for the ambient boundary $\partial N$.	
		 \end{rmk}	
	\noindent With that notation, recall that
		\[
		\kappa=g(D_{\tau}\nu,\tau)=-II(\tau,\tau)
		\]	
		where $D$ denotes the Levi-Civita connection on $N$, $\tau$ is a tangent vector to $\partial M\subset\partial N$ at the point in question, and the last equality relies on our sign conventions concerning the second fundamental form $II$ of $\partial N\subset N$ (which are consistent with \cite{ACS17}).  Thus, considering a smooth extension of this function (which is a priori only defined on the subbundle whose base is $\partial M\subset\partial N$), Proposition \ref{prop:boundaryconv} has the following geometric consequence:
		
		\begin{cor}\label{cor:geodcurv}
		In the setting of Theorem \ref{thm:quant} specified to ambient dimension three, we have	
		\[
		\lim_{k\to\infty}\int_{\partial M_k} \kappa(M_k)\,d\h^1 = m\int_{\partial M} \kappa(M)\,d\h^1.
		\]
		\end{cor}
		
		Hence, combining the quantization identity with the Gauss-Bonnet theorem, Lemma \ref{lem:euler} and Corollary \ref{cor:geodcurv} (which ensures that the geodesic boundary terms cancel out) we obtain a proof of Corollary \ref{cor:chi} stated in the introduction: the equation
		\begin{equation*}
		\chi(M_k)= m\chi(M)+ \sum_{j=1}^J (\chi(\Sigma_j)-b_j),
		\end{equation*}
 holds for all $k$ sufficiently large.

		We now prove our three geometric results, Theorem \ref{thm:conv1}, Theorem \ref{thm:conv2} and Theorem \ref{thm:lsc}.

		\begin{proof}[Proof of Theorem \ref{thm:conv1}] Let $\left\{M_k\right\}$ be a sequence of embedded, free boundary minimal surfaces in $(N,g)$. In either case we are considering, Theorem 2 in \cite{ACS17} applies to provide subsequential convergence (smooth away from at most finitely many points).
		In the stable case, there is no bubbling and thus (possibly extracting a subsequence, which we shall not rename), we get 
		\begin{equation}\label{eq:chi0}
		\chi(M_k)=m\chi(M)
		\end{equation}
		for $k$ large enough. Now, if the convergence happened with multiplicity $m\geq 2$ then (by virtue of Lemma \ref{lem:stable-type}) the limit $M$ would be a disk and thus the right-hand side would be greater or equal than $2$ but on the other hand the Euler characteristic of any connected surface with boundary is at most one (with equality only in the case of the disk). Thus, equation \eqref{eq:chi0} gives a contradiction unless the convergence is smooth with multiplicity one.
		
		Let us now discuss the index one case instead. The case when the convergence is smooth, but with multiplicity, is dealt with as we did for the stable case. Else, let us consider the case when bubbling occurs: by Theorem \ref{thm:quant} there would be a non-trivial bubble (or a non-trivial half-bubble) and by the index estimate given by Lemma 12 in \cite{ABCS18} (which also holds here as a direct consequence of Theorem \ref{thm:bubbles}), it must have index at most one. Therefore, our classification results (specifically: Corollary \ref{cor:indexone} for half-bubbles) imply that the surface in question is either a catenoid or a vertically cut half-catenoid. Again, the limit surface must be a free boundary stable minimal disk. We now use this information in the identity given in Corollary \ref{cor:chi} and get the bounds
		\[
		\chi(M_k)\geq 0 \ \ \text{or} \ \ \chi(M_k)\geq 1
		\]
		depending on whether the first or the second alternative occurs, respectively. Hence we get strong multiplicity one convergence provided we assume $\chi(M_k)<0$ for all $k$.

		Let us now examine the cases when the Euler characteristic of the surfaces in our sequence equals zero or one. The smooth convergence of the local rescalings, as described by Theorem \ref{thm:bubbles} implies (by the multiplicity estimate given by Proposition 13 in \cite{ABCS18}, which can be transplanted here with the very same proof) that $m\leq 2$ and hence $m=2$.
	If we apply Corollary \ref{cor:chi}, we see at once that for $k\in\mathbb{N}$ sufficiently large $\chi(M_k)=0$ implies the presence of a catenoidal bubble, and instead $\chi(M_k)=1$ implies a vertically cut catenoidal half-bubble. Yet, we claim that the former case is only possible when one considers surfaces having the topological type of an annulus, but not in the case of M\"obius bands. Let us see why.
	 Eventually each surface $M_k$ shall lie in a given tubular neighborhood of the limit surface $M$, which is a disk (hence two-sided, cf. Lemma \ref{lem:stable-type}). Such a tubular neighborhood is then diffeomorphic to the product of the disk times an open interval $I$, hence it cannot contain any proper one-sided surface with boundary on $\partial D\times I$ and thus in particular it cannot contain any free boundary minimal M\"obius band.

		Let us then justify the final assertion in the statement of the theorem. Under the topological assumption that $N$ be simply connected, we know that all free boundary minimal surfaces it contains must be two-sided (hence orientable themselves). It follows that Lemma \ref{lem:area} applies to provide area estimates, and at that stage one can just follow the argument above.
	\end{proof}

    	\begin{proof} [Proof of Theorem \ref{thm:conv2}]
    	Possibly by extracting a subsequence, which we shall not rename, we can assume (which will be always implicit in the sequel of this proof) that
    	\[
    	\lim_{k\to\infty}\h^2(M_k)> \frac{8 \pi}{\varrho}
    	\]
    	or, respectively,
    	\[
    	\lim_{k\to\infty}\h^1(\partial M_k)>\frac{4\pi}{\sigma}
    	\]
    	and, like in the proof of Theorem \ref{thm:conv1}, that the sequence $M_k$ converges to some limit minimal surface $M$ as described by Theorem \ref{thm:quant}.
   
  Arguing by contradiction, thus assuming that the convergence is not smooth with multiplicity one,  by Theorem \ref{thm:quant} and Lemma 12 in \cite{ABCS18} there can be at most one non-trivial bubble or a non-trivial half-bubble, and that must be a catenoid or a half-catenoid. If the convergence is smooth but there is no bubbling then $m=2$ (we have stability when passing to the double cover) hence, thanks to Lemma \ref{lem:stable-type} we get
    	\begin{equation}\label{eq:areaexc}
    	\lim_{k\to\infty}\h^2(M_k)\leq \frac{8 \pi}{\varrho}
    	\end{equation}
    	or, respectively,
    	\begin{equation}\label{eq:lengthexc}
    	\lim_{k\to\infty}\h^1(\partial M_k)\leq \frac{4\pi}{\sigma}.
    	\end{equation}
    	If instead bubbling occurs, the characterization of a two-sided disk as limit surface (cf. Lemma \ref{lem:stable-type}) and of the catenoid or vertically cut half-catenoid as blow-ups, allow to employ the multiplicity estimate
    	(Proposition 13 in \cite{ABCS18})
    	to gain $m\leq 2$, hence necessarily $m=2$.
    	By varifold convergence (at the interior, see \cite{ACS17} and at the boundary, by Proposition \ref{prop:boundaryconv}) we get the same area/length bounds as the two cases above, namely \eqref{eq:areaexc} or \eqref{eq:lengthexc} respectively.
    This is incompatible with our postulated area/length bounds, so it must be $m=1$.	
    			\end{proof}

    			\begin{proof}[Proof of Theorem \ref{thm:lsc}]
    		    The inequality follows by combining Corollary \ref{cor:chi} with the multiplicity estimate, Proposition 13 in \cite{ABCS18}.
    				In case equality occurs, in particular we must have equality in the multiplicity estimate so
    				\[
    				m=1+\sum_{j=1}^J(b_j-1)
    				\]
    				and the properness assumption $(\textbf{P})$ implies (once again, see Subsection \ref{subs:geomhalf} ) that $b_j=\check{b}_j\geq 2$, where the last inequality relies on the fact that the plane is the only complete, embedded minimal surface in $\R^3$ of finite total curvature having one end (cf. \cite{Sc83}, but this would also follow from the half-space theorem of Hoffman and Meeks \cite{HM90}). Hence there are at most $m-1$ summands, each corresponding either to a non-trivial bubble or to a non-trivial half-bubble. If there are exactly $m-1$ summands then $b_j=\check{b}_j=2$ for each $j=1,\ldots, J$ and the conclusion comes again from the characterization of the catenoid as the only complete, embedded minimal surface in $\R^3$ with exactly two ends, also contained in \cite{Sc83}.
    			\end{proof}	
    
    	\begin{rmk}
    		If one drops assumption (\textbf{P}), thereby allowing improper free boundary minimal surfaces, then the equality case might (at least in principle) allow for a larger number of half-bubbles, by virtue of the potential presence of horizontally cut half-catenoids (cf. Corollary \ref{cor:oneend}).
    	\end{rmk}

    \appendix
    
    \section{Area estimates}\label{sec:area}
    
    The next lemma collects some known estimates for free boundary minimal surfaces of index zero or one:
    
    \begin{lem}\label{lem:stable}
    	Let $(N^3,g)$ be a compact Riemannian manifold, with non-empty boundary $\partial N$ and let $M\subset N$ be a properly embedded, two-sided, free boundary minimal surface. Let $\varrho:=\inf R$ (where $R$ denotes the scalar curvature of $(N,g)$) and $\sigma:=\inf  H$ (where $H$ denotes the mean curvature of $\partial N\subset N$), both assumed to be non-negative numbers.
    	\begin{enumerate}
    	\item{If $M$ is stable, then
    		\[
    		\frac{\varrho}{2}\h^2(M)+\sigma\h^1(\partial M)\leq 2\pi\chi(M)\leq 2\pi.
    		\]}	
    	\item{If $M$ has index one and is orientable, then
    		\[
    		\frac{\varrho}{2}\h^2(M)+\sigma\h^1(\partial M)\leq 2\pi (8-\text{boundaries}(M))\leq 16\pi.
    		\]}	
    	\end{enumerate}	
    \end{lem}

    \begin{proof}
    	For part (1), the well-known Schoen-Yau rearrangement trick, the Gauss-Bonnet theorem and the free boundary condition allow to derive from the stability inequality the estimate
    		\begin{equation}\label{eq:STAB}
    		\frac{1}{2}\int_{M}(R+|A|^2)\,d\h^2+\int_{\partial M}H\,d\h^1\leq 2\pi\chi(M)
    		\end{equation}
    	where $R$ indicates the scalar curvature of $(N,g)$,  $A$ is the second fundamental form of $M\subset N$ and $H$ is the (ambient) mean curvature of $\partial N\subset N$. However, 
    	\begin{equation}\label{eq:chiFBMS}
    	\chi(M)=
    	\begin{cases}
    	2-2genus(M)-boundaries(M) & \text{if} \ M \ \text{is orientable}; \\
    	1-genus(M)-boundaries(M) & \text{if} \ M \ \text{is not orientable}
    	\end{cases}
    	\end{equation}
    	and thereby the conclusion follows at once.

    	  The proof of part (2) relies, instead, on the well-known \textsl{Hersch trick}. For completeness, we are going to outline the argument here.
    Throughout this proof, we let $\gamma$ denote the genus of $M$ and $r$ the number of its boundary components, furthermore let $\lambda_1<0$ be the first eigenvalue of the Jacobi operator 
    \[
    L\theta=\Delta_{M}\theta + (Ric(\epsilon,\epsilon)+|A|^2)\theta
    \]
     and $\theta_1>0$ an associated positive eigenfunction, so that
    	\[
    	\begin{cases}
    	L\theta_1=-\lambda_1\theta_1 & \text{on} \ M;\\
   \hspace{1mm}	\frac{\partial\theta_1}{\partial\nu} =-II(\epsilon,\epsilon)\theta_1 & \text{on} \ M.
    	\end{cases}
    	\]
    	By the index one assumption, considered the stability form
    	\[
    	Q(\theta,\theta)= \int_{M} \left(|\nabla \theta|^2 - (Ric_N(\epsilon,\epsilon) + |A|^2)\theta^2 \right)\id\h^{n} + \int_{\partial M}{\rm{II}}(\epsilon,\epsilon)\theta^2 \id \h^{n-1}
    	\]
    	 we have that $Q(\phi,\phi)\geq 0$ for all $\phi\in C^{\infty}(M)$ satisfying $\int_M \phi ~\theta_1\,d\h^2=0$. In order to choose appropriate test functions, we follow the approach suggested, to a somewhat different purpose, by Ros and Vergasta in \cite{RV95} (see also \cite{CFP14}, Theorem 5.1). Let us cap each boundary component of $M$ with a closed disk and consider a smooth extension of the metric to obtain $(\overline{M},\overline{g})$ a closed orientable surface of genus $\gamma$: by Riemann-Roch there exists a holomorphic map $\overline{\phi}:(\overline{M},\overline{g})\to(S^2,g_0)$ of degree bounded from above by $\floor{\frac{\gamma+3}{2}}$, and possibly by composing with a conformal diffeomorphism of the two-sphere with its round metric $g_0$ we can assume that indeed each of its components satisfies the orthogonality relation $\int_M \overline{\phi}_i \theta_1\,d\h^2=0$ for $i=1,2,3$ where we are identifying such $S^2$ with the unit sphere in $\R^3$. As a result, considering the stability inequalities for each of these three functions and adding them we get
    	\[
    	\frac{1}{2}\int_{M}(R+|A|^2)\,d\h^2+\int_{\partial M}H\,d\h^1\leq \sum_{i=1,2,3} \int_M |\nabla\overline{\phi}_i|^2\,d\h^2+ 2\pi\chi(M)
    	\]
    	where we have used the fact that $\overline{\phi}^2_1+\overline{\phi}^2_2+\overline{\phi}^2_3=1$. At that stage, we can estimate the first summand on the right-hand side in terms of the degree of $\overline{\phi}$:
    	\[
    	\sum_{i=1,2,3} \int_M |\nabla\overline{\phi}_i|^2\,d\h^2\leq \sum_{i=1,2,3} \int_{\overline{M}} |\nabla\overline{\phi}_i|^2\,d\h^2=\text{deg}(\overline{\phi})\sum_{i=1,2,3}\int_{S^2}|(d x_i)^T|^2\,d\h^2=8\pi \text{deg}(\overline{\phi})
    	\]
    	where $(d x_i)^T$ denotes the component of the Euclidean one-form $dx_i$ that is tangent to round $S^2$; hence we derive
    	\[
    	\frac{1}{2}\int_{M}(R+|A|^2)\,d\h^2+\int_{\partial M}H\,d\h^1\leq 8\pi \text{deg}(\overline{\phi})+2\pi\chi(M).
    	\] Expressing the Euler characteristic and the degree of this map in terms of the topological data of $M$ we get
    	\[
    	\frac{\varrho}{2} \h^2(M)+\sigma \h^1(\partial M)\leq 2\pi \cdot
    	\begin{cases}
    	4(k+1)+2-2(2k)-r & = 6 -r  \ \ \ \ \text{if} \ \gamma=2k\\
    	4(k+1)+2-2(2k-1)-r & = 8 -r  \ \ \ \ \text{if} \ \gamma=2k-1\\
    	\end{cases}
    	\]
    	hence in particular the left-hand side is always bounded from above by $16\pi$, which completes the proof.
    \end{proof}	
    
     Also, we will employ the following area bounds for index one free boundary minimal surfaces.
     
     \begin{lem}\label{lem:area}
     	Let $(N^3,g)$ be a compact Riemannian manifold, with non-empty boundary $\partial N$ and let $M\subset N$ be a connected, properly embedded, two-sided and orientable free boundary minimal surface of index 0 or 1.
     	Assume that: 
     	\begin{itemize} 
     		\item{\underline{\textsl{either}} the scalar curvature of $(N,g)$ is positive and $\partial N$ is mean convex;}
     		\item{\underline{\textsl{or}} the scalar curvature of $(N,g)$ is non-negative, $\partial N$ is strictly mean convex and there is no \emph{closed} minimal surface in $N$.}
     	\end{itemize}
     	Then there is an upper bound of the area of $M$ which only depends on the manifold $(N,g)$.
     \end{lem}	
     
     \begin{proof}
     Consider first the case of stable surfaces. Under the first assumption we get from Lemma \ref{lem:stable} that the area of $M$ is bounded from above by $4\pi/\varrho$. Under the second assumption, we get that the length of $\partial M$ is bounded from above by $2\pi/\sigma$, and then the conclusion comes 	
     by invoking the isoperimetric inequality by White, see Theorem 2.1 in \cite{Whi09}. The same argument applies to index one surfaces.
     \end{proof}

    In studying those free boundary minimal surfaces arising as higher-multiplicity limits we shall need the following classification result, which is the free boundary analogue of Lemma 14 in \cite{ABCS18}.

    \begin{lem}\label{lem:stable-type}
    	Let $(N^3,g)$ be a compact Riemannian manifold, with non-empty boundary $\partial N$ and let $M\subset N$ be a properly embedded, free boundary minimal surface arising as the limit (in the sense of smooth convergence with multiplicity $m\geq 2$ away from a finite set $\mathcal{Y}$ of points) of a sequence of embedded, free boundary minimal surfaces in $(N,g)$. 
    	Assume that: 
    	\begin{itemize} 
    		\item{\underline{\textsl{either}} the scalar curvature of $(N,g)$ is positive and $\partial N$ is mean convex;}
    		\item{\underline{\textsl{or}} the scalar curvature of $(N,g)$ is non-negative and $\partial N$ is strictly mean convex.}
    	\end{itemize}
    	Then $M$ is two-sided and diffeomorphic to a disk $D^2$ satisfying the geometric bound
    	\[
    	\frac{\varrho}{2}\h^2(M)+\sigma\h^1(\partial M)\leq 2\pi.
    	\]
    \end{lem}

    \begin{proof}
    	Let us first consider the case when $M$ is two-sided: the geometric assumption we made implies that $M$ must be stable (cf. Section 5 of \cite{ACS17}) and thus the stability inequality, cf. \eqref{eq:STAB}, together with \eqref{eq:chiFBMS}
    	imply that it has positive Euler characteristic so that it must be a disk.
    	We now claim that the case when $M$ is one-sided \textsl{does not occur}.  If $M$ were one-sided, we could consider the construction of the twofold cover $\widetilde{M}\subset\widetilde{U}$ (as discussed, for the free boundary case, in Section 6 of \cite{ACS17}).
    	The local picture around any given point is unchanged, so one still has that the convergence of $\widetilde{M}_k$ to $\widetilde{M}$ happens with the same multiplicity $m\geq 2$. Thus $\widetilde{M}$ would be stable and the argument above would apply. Hence $\widetilde{M}$ would be a disk, but a disk is not the double cover of any compact surface with boundary because any (continuous) automorphism of the disk has a fixed point. Thus, this contradiction proves the claim.
    	Lastly, the geometric bounds above comes directly from case (1) of Lemma \ref{lem:stable}.
    \end{proof}

  \end{document}